\documentclass[11pt]{amsart}
\usepackage{amscd,amssymb}
\usepackage[all]{xy}
\usepackage{graphicx,calrsfs,verbatim} 
\usepackage[colorlinks,plainpages,backref,urlcolor=blue,breaklinks=true]{hyperref}
\usepackage{bbm}
\usepackage{mathrsfs}
\usepackage{xstring}
\usepackage{tikz}
\usetikzlibrary{matrix,arrows,decorations.pathmorphing,calc,shapes}
\usepackage{array}
\usepackage{breakurl}

 \newtheorem{theorem}{Theorem}[section]
\newtheorem{corollary}[theorem]{Corollary}
\newtheorem{lemma}[theorem]{Lemma}
\newtheorem{prop}[theorem]{Proposition}

\theoremstyle{definition}

\newtheorem{example}[theorem]{Example}
\newtheorem{remark}[theorem]{Remark}

\newtheorem{question}[theorem]{Question}

\DeclareMathOperator{\im}{im}
\DeclareMathOperator{\initial}{in}
\DeclareMathOperator{\id}{id}

\DeclareMathOperator{\Ann}{Ann}
\DeclareMathOperator{\rank}{rank}
\DeclareMathOperator{\Ext}{Ext}

\DeclareMathOperator{\ab}{ab}

\DeclareMathOperator{\gr}{gr}
\DeclareMathOperator{\Sym}{Sym}
\DeclareMathOperator{\Hilb}{Hilb} 
\DeclareMathOperator{\Poin}{Poin}

\DeclareMathOperator{\Tr}{Tr}
\DeclareMathOperator{\QTr}{QTr}

\DeclareMathOperator{\Conf}{Conf}

\newcommand{\R}{\mathbb{R}}

\newcommand{\C}{\mathbb{C}}
\newcommand{\Z}{\mathbb{Z}}

\newcommand{\bG}{\mathbf{G}}

\newcommand{\fB}{\mathfrak{B}}
\newcommand{\fC}{\mathfrak{C}}

\newcommand{\fh}{\mathfrak{h}}
\newcommand{\fg}{\mathfrak{g}}
\newcommand{\g}{\mathfrak{g}}

\newcommand{\fm}{\mathfrak{m}}
\newcommand{\fL}{\mathfrak{L}}

\newcommand{\h}{\mathfrak{h}}

\newcommand{\bL}{\mathbf{L}}

\newcommand{\cA}{\mathcal{A}}

\newcommand{\cR}{\mathcal{R}}

\newcommand{\vB}{vB}
\newcommand{\vP}{vP}

\newcommand{\PV}{vP}

\newcommand{\surj}{\twoheadrightarrow}
\newcommand{\inj}{\hookrightarrow}

\def\dot{\mathchar"013A"}  
\newcommand{\hdot}{{\raise1pt\hbox to 0.35em{\Huge $\dot$}}}

\begin{document}


\title[Pure virtual braids, resonance, and formality]%
{Pure virtual braids, resonance, and formality}

\author[Alexander~I.~Suciu]{Alexander~I.~Suciu$^1$}
\address{Department of Mathematics,
Northeastern University,
Boston, MA 02115, USA}
\email{\href{mailto:a.suciu@neu.edu}{a.suciu@neu.edu}}
\urladdr{\href{http://www.northeastern.edu/suciu/}%
{www.northeastern.edu/suciu/}}
\thanks{$^1$Supported in part by the National Security 
Agency (grant H98230-13-1-0225) and the Simons Foundation 
(collaboration grant for mathematicians 354156)}

\author{He Wang}
\address{Department of Mathematics,
Northeastern University,
Boston, MA 02115, USA}
\curraddr{Department of Mathematics and Statistics,
University of Nevada, Reno, NV 89557, USA}
\email{\href{mailto:wanghemath@gmail.com}%
{wanghemath@gmail.com}}
\urladdr{\href{http://wolfweb.unr.edu/homepage/hew/}%
{http://wolfweb.unr.edu/homepage/hew/}}



\begin{abstract}
We investigate the resonance varieties, lower central series ranks, 
and Chen ranks of the pure virtual braid groups and their upper-triangular 
subgroups.   As an application, we give a complete answer to the 
$1$-formality question for this class of groups.  In the process, 
we explore  various connections between the Alexander-type 
invariants of a finitely generated group and several of the graded Lie algebras 
associated to it, and discuss possible extensions of the resonance-Chen 
ranks formula in this context.
\end{abstract}

\maketitle
\setcounter{tocdepth}{1}
\tableofcontents

\section{Introduction}
\label{sect:intro}

\subsection{Background}
\label{intro:pvb}
Virtual knot theory, as introduced by Kauffman in \cite{Kauffman99}, 
is an extension of classical knot theory. This new theory studies embeddings 
of knots in thickened surfaces of arbitrary genus, while the classical 
theory studies the embeddings of circles in thickened spheres. 
Another motivation comes from the representation of knots by 
Gauss diagrams.  In \cite{Goussarov-P-V00}, Goussarov, Polyak, and Viro
showed that the usual knot theory embeds into virtual knot theory, 
by realizing any Gauss diagram by a virtual knot.
Many knot invariants, such as knot groups, the bracket polynomial, 
and finite-type Vassiliev invariants can be extended to invariants 
of virtual knots,  see \cite{Kauffman99, Goussarov-P-V00}.

The virtual braid groups $\vB_n$ were introduced in \cite{Kauffman99} and
further studied in \cite{Bar-Natan14,Bardakov-B14, Kauffman-L07, 
Bardakov04,Kamada07}. As shown by Kamada in \cite{Kamada07},  
any virtual link can be constructed as the closure of a virtual braid,
which is unique up to certain Reidemeister-type moves.
In this paper, we will be mostly interested in the kernel of the canonical 
epimorphism $\vB_n\to S_n$, called the {\em pure virtual 
braid group}, $vP_n$, and a certain subgroup of this group, $\vP_n^+$, 
which we call the {\em upper pure virtual braid group}.

In \cite{Bartholdi-E-E-R}, Bartholdi, Enriquez, Etingof, and Rains 
independently defined and studied the group $\PV_n$, which they 
called the $n$-th quasitriangular groups $\QTr_n$, as a group-theoretic 
version of the set of solutions to the Yang--Baxter equations. 
Their work was developed in a deep way by P.~Lee in \cite{Lee}.  
The authors of \cite{Bartholdi-E-E-R, Lee} construct a 
classifying space for $\vP_n$ with finitely many cells, and find a presentation 
for the cohomology algebra of $\vP_n$, which they show is a Koszul algebra. 
They also obtain parallel results for a quotient group of $\vP_n$, called the 
the $n$-th triangular group, which has the same generators as $\vP_n$, 
and one more set of relations. It is readily seen that the triangular group 
$\Tr_n$ is isomorphic to $\vP_n^+$.

We refer to \cite{SW2} for further context on the pure virtual braid groups and 
related subgroups of basis-conjugating automorphisms of free groups.  

\subsection{Presentations and associated graded Lie algebras}
\label{intro:pres} 

Bardakov gave in \cite{Bardakov04} a presentation for the pure virtual braid group
$\vP_n$, much simpler than the usual presentation of the pure braid group $P_n$.
As shown in \cite{Bartholdi-E-E-R}, there exists a monomorphism from $P_n$ to $\vP_n$.
Moreover, there are split injections 
$\PV_n\rightarrow \PV_{n+1}$, $\PV_n^+\rightarrow \PV_{n+1}^+$ and 
$\PV_n^+\rightarrow \PV_{n}$. 

The pure braid group $P_n$ has center $\Z$, so there is a decomposition 
$P_n\cong \overline{P}_n\times \Z$. Using a decomposition of $\vP_3$ given 
by Bardakov, Mikhailov, Vershinin, and Wu in \cite{Bardakov-M-V-W}, 
we show  that $\vP_3\cong \overline{P}_4\ast \Z$.   In this context, it is 
worth noting that the center of $\vP_n$ is trivial for $n\geq 2$, and
the center of $\vP_n^+$ is trivial for $n\geq 3$, 
with one possible exception; see Dies and Nicas  \cite{Dies-Nicas14}.

In a different vein, Labute \cite{Labute85} and Anick \cite{Anick87} defined 
the notion of a `mild' presentation for a finitely presented group $G$. 
If the group $G$ admits such a presentation, then a presentation 
for the (complex) associated graded Lie algebra $\gr(G)$ 
can be obtained from the classical Magnus expansion.
In general, though, finding a presentation for this Lie algebra is 
an onerous task.  In \S\ref{subsec:mild}, we prove  the following result.

\begin{prop}
\label{prop:intro}
The pure braid groups $P_n$ and the pure virtual braid groups 
$\vP_n$ and $\vP_n^+$  admit mild presentations if and only $n\leq 3$. 
\end{prop}

Nevertheless, explicit, quadratic presentations for the associated 
graded Lie algebras of the groups $G_n=P_n$, $\vP_n$, 
and $\vP_n^+$ were given in \cite{Kohno85, Falk-Randell85} 
for $P_n$, and in \cite{Bartholdi-E-E-R, Lee} for $\vP_n$ and $\vP_n^+$.  
These computations also show that, for each of these groups, 
the universal enveloping algebra $U(\gr(G_n))$ is a Koszul algebra.
Using Koszul duality and some combinatorial manipulations, we 
show in \S\ref{subsec:coho ring vpb} 
that the LCS ranks of the groups $G_n$ are given by 
\begin{equation}
\label{eq:phi ranks}
\phi_k(G_n)=\dfrac{1}{k} \sum_{d|k}\mu\left(\frac{k}{d}\right)\left[
\sum_{m_1+2m_2+\cdots+nm_n=d} (-1)^{s_n}  d(m!) 
\prod_{j=1}^{n} \frac{(b_{n,n-j})^{m_j}}{ (m_j) !}  \right],
\end{equation}
where $m_j$ are non-negative integers, $s_n=\sum_{i=1}^{[n/2]}m_{2i}$, 
$m=\sum_{i=1}^{n}m_{i}-1$, and $\mu$ is the M\"{o}bius function,
while $b_{n,j}$ are the (unsigned) Stirling numbers of the first kind 
(for $G_n=P_n$), the Lah numbers (for $G_n=\vP_n$), or the Stirling 
numbers of the second kind (for $G_n=\vP_n^+$).

\subsection{Resonance and formality}
\label{intro:formal}

Consider now a group $G$ admitting a finite-type classifying space 
$K(G,1)$.   The cohomology algebra $H^*(G,\C)$ may be turned  
into a family of cochain complexes parametrized by the affine space    
$H^1(G,\C)$, from which one may define the  {\em resonance 
varieties}\/ of the group, $\cR^i_d(G)$, as the loci where 
the cohomology of those cochain complexes jumps.   
After a brief review of these notions, we study in \S\ref{sec:res} the behavior 
of resonance under product and coproducts, obtaining formulas which 
generalize those from \cite{Papadima-Suciu10P}, see 
Propositions \ref{prop:resProd} and \ref{prop:res}.

The formality and partial formality properties of spaces and groups are 
basic notions in homotopy theory, allowing one to describe the rational 
homotopy type of a simply-connected space, or the tower of nilpotent 
quotients of a group in terms of the rational cohomology ring of the 
object in question.  More concretely, a finitely generated group $G$ 
is said to be {\em $1$-formal}\/ if its Malcev Lie algebra admits a 
quadratic presentation. 
Following \cite{SW1}, we will separate the $1$-formality property 
into graded-formality and filtered-formality.

The work of Bartholdi et al.~\cite{Bartholdi-E-E-R} and Lee \cite{Lee} 
mentioned above shows that the pure virtual braid group $\PV_n$ 
and its subgroup $\PV^+_n$ are graded-formal, for all $n$.
Furthermore, Bartholdi, Enriquez, Etingof, and Rains state 
that the groups $\PV_n$ and $\PV^+_n$ are not $1$-formal for $n\geq 4$, 
and sketch a proof of this claim.  One of the aims of this paper (indeed, 
the original motivation for this work) is to provide a detailed proof of 
this fact.  Our main result, which is proved in \S \ref{sec: Formality}, 
reads as follows.

\begin{theorem}
\label{thm:intro formal}
The groups $\PV_n$ and $\PV^+_n$ are both $1$-formal if $n\le 3$, 
and they are both non-$1$-formal (and thus, not filtered formal) if $n\ge 4$.
\end{theorem}

The $1$-formality property of groups is preserved under split injections 
and (co)products, see \cite{SW1, DPS}.  
Consequently, the fact that we have split injections between the various 
pure virtual braid groups allows us to reduce the proof of 
Theorem \ref{thm:intro formal} to verifying the $1$-formality 
of $\vP_3$ and the non-$1$-formality of $\vP_4^+$.
To prove the first statement, we use the free product decomposition 
$\vP_3\cong \Z\ast \overline{P}_4$.  For the second statement, 
we compute the resonance variety $\cR^1_1(\PV_4^+)$, and use 
the geometry of this variety, together 
with the Tangent Cone Theorem from \cite{DPS} 
to reach the desired conclusion.

\subsection{Chen ranks and Alexander invariants}
\label{intro:alexinv}

Given a finitely generated group $G$, we let 
$\theta_k(G)=\dim \gr_k(G/G^{\prime\prime})$ be the LCS ranks of the maximal 
metabelian quotient of $G$.  These ranks were first studied 
by K.-T. Chen in \cite{Chen51}, and are named after him. In \S\ref{sec:Chen}, 
we give the generating functions for the Chen ranks of the free groups $F_n$
and the pure braid groups $P_n$, based on computations from 
\cite{Chen51, Cohen-Suciu95}. 

As shown by W. Massey in \cite{Massey80}, the Chen ranks can be computed 
from the Alexander invariant, $B(G)=G^{\prime}/G^{\prime\prime}$.  More precisely, 
if we view this abelian group as a module over $\C[H]$, where $H=G/G'$, then 
filter it by powers of the augmentation ideal, and  take the 
associated graded module, $\gr(B(G))$, viewed as a module over 
the symmetric algebra $S=\Sym(H\otimes \C)$, we have that 
$\theta_{k+2}(G)=\dim \gr_{k}(B(G))$ for all $k\geq 0$.  

In a similar fashion, we define the Chen ranks of a finitely generated, 
graded Lie algebra $\fg$ to be $\theta_k(\fg)=\dim (\fg/\fg'')_k$, and
the infinitesimal Alexander invariant of $\fg$ to be the graded 
$\Sym(\fg_1)$-module $\fB(\fg)=\fg^{\prime}/\fg^{\prime\prime}$, 
after which we show that $\theta_{k+2}(\fg)=\dim \fB(\fg)_k$ for all $k\geq 0$.  

Our next main result (a combination of Propositions \ref{prop:alexinv compare}, 
\ref{prop:gamma compare}, and \ref{thm:gradedinf}), relates the various Alexander-type 
invariants associated to a group, as follows.

\begin{theorem}
\label{thm:compare}
Let $G$ be a finitely generated group with abelianization $H$, 
and set $S=\Sym(H\otimes \C)$.  There exists then surjective morphisms 
of graded $S$-modules,
\begin{equation}
\label{eq:surj bees}
\xymatrix{ \fB(\fh(G)) \ar@{->>}[r]^{\psi} &\fB(\gr(G)) \ar@{->>}[r]^{\varphi} &\gr(B(G))}.
\end{equation} 
Moreover, if $G$ is graded-formal, then $\psi$ is an isomorphism, and 
if $G$ is filtered-formal, then $\varphi$ is an isomorphism.
\end{theorem}

This result yields the following inequalities between the various types of 
Chen ranks associated to a finitely generated  group $G$:
\begin{equation}
\label{eq:theta ineq}
\theta_k(\fh(G)) \geq \theta_k(\gr(G)) \geq \theta_k(G), 
\end{equation}
with  the first inequality holding as equality if $G$ is graded-formal, 
and the second inequality holding as equality if $G$ is filtered-formal.

\subsection{The Chen ranks formula}
\label{intro:res chen}

Now suppose $G$ is a finitely presented, commutator-relators 
group.  Then, as shown in \cite{Papadima-Suciu04}, the infinitesimal 
Alexander invariant $\fB(\fh(G))$ coincides with the `linearization' of the Alexander 
invariant $B(G)$.  Furthermore, as shown in \cite{Matei-Suciu00}, the 
resonance variety $\cR_1^1(G)$ coincides, away from the origin $0\in H^1(G;\C)$, 
with the support variety of the annihilator of $\fB(\fh(G))$.

Under certain conditions, the \emph{Chen ranks formula}\/ reveals a close 
relationship between the first resonance variety of $G$ and the Chen ranks,
namely, 
\begin{equation}
\label{eq:cranks}
\theta_k(G)= \sum_{m\geq 2}h_m(G)\cdot \theta_k(F_m)
\end{equation}
for $k\gg 1$, 
where $h_m(G)$ is the number of $m$-dimensional irreducible 
components of $\cR_1^1(G)$.  This formula was conjectured in \cite{Suciu01} 
for arrangement groups and proved by Cohen and Schenck in \cite{Cohen-Schenck15} 
for $1$-formal groups satisfying certain restrictions on their resonance components. 

In \S\ref{sec:Chen}, we analyze the Chen ranks formula in a wider setting, 
with a view towards comparing the Chen ranks and the resonance varieties 
of the pure virtual braid groups.   We start by noting that formula \eqref{eq:cranks} 
may hold even for non-$1$-formal groups, such as the fundamental groups of 
complements of suitably chosen arrangements of planes in $\R^4$.  

Next, we look at the way the Chen ranks formula behaves well with respect 
to products and coproducts of groups.  The conclusion may be summarized as follows.  

\begin{prop}
\label{prop:introproductChenformula}
If both $G_1$ and $G_2$ satisfy the Chen ranks formula \eqref{eq:cranks},
then $G_1\times G_2$ also satisfies the Chen ranks formula, but 
$G_1* G_2$ may not.
\end{prop}

Finally, we analyze the Chen ranks and the degree $1$ resonance 
varieties of the groups $\PV_n$ and $\PV^+_n$.  
In a preprint version of \cite{Bardakov-M-V-W},  
Bardakov et al.~compute some of the 
degree $1$ and $2$ resonance varieties of $\PV_3$. 
Using the decomposition $\vP_3\cong \Z\ast \overline{P}_4$, 
we compute all the resonance varieties of  this group.  We also compute 
the ranks $\theta_k(\vP_3)$, and show that formula \eqref{eq:cranks} 
does not hold in this case, although $vP_3$ satisfies all but one of the 
hypothesis of \cite[Theorem A]{Cohen-Schenck15}.  
Likewise, we compute the Chen ranks and degree $1$ resonance 
varieties of the groups $\vP_n^+$ for low values of $n$, and 
conclude that $\vP_4^+$ and $\vP_5^+$ also do not 
satisfy the Chen ranks formula. 

We would like to thank the referee for helpful comments and suggestions. 

\section{Pure virtual braid groups}
\label{sect:triangular}

In this section, we look at the pure virtual braid groups and their 
upper-triangular subgroups from the point view of combinatorial 
group theory.

\subsection{Presentations}
\label{subsec:presentations}

Let $P_n$ be the Artin pure braid group on $n$ strings.  
As is well-known, the center of $P_n$ ($n\ge 2$) is infinite cyclic, and so 
we have a direct product decomposition of the form 
$P_n\cong \overline{P}_n\times \Z$.  
The first few groups in this series are easy to describe: 
$P_1=\{1\}$, $P_2=\Z$, and $P_3\cong F_2\times \Z$, 
where $F_n$ denotes the free group on $n$ generators.

As shown in \cite{Bardakov04}, the \emph{pure virtual braid group}\/ $\PV_n$ 
has presentation with generators
$x_{ij}$ with $1\leq i\neq j \leq n$ (see Figure \ref{fig:purebraids} for a  
description of the corresponding virtual braids), and relations  
\begin{align} 
&x_{ij}x_{ik}x_{jk}=x_{jk}x_{ik}x_{ij} &&\text{for distinct $i,j,k$}, \\  \notag
&x_{ij}x_{kl}=x_{kl}x_{ij} &&\text{for distinct $i,j,k,l$}.
\end{align}

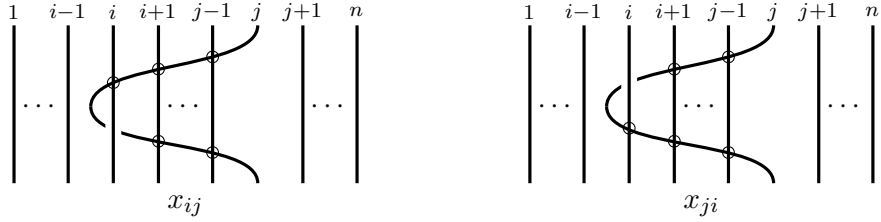
\begin{figure}[t]
\centering
\begin{tikzpicture}[scale=0.6]
\draw (1, 0.3) node {$_{1}$};
\draw (2.2, 0.3) node {$_{i-1}$};
\draw (3.2, 0.3) node {$_{i}$};
\draw (4.2, 0.3) node {$_{i+1}$};
\draw (5.4, 0.3) node {$_{j-1}$};
\draw (6.4, 0.3) node {$_{j}$};
\draw (7.4, 0.3) node {$_{j+1}$};
\draw (8.6, 0.3) node {$_{n}$};
\draw (1.6, -1.8) node {$\cdots$};
\draw (8, -1.8) node {$\cdots$};
\draw (4.8, -1.8) node {$\cdots$};
 
\draw[ very thick] (1,0) -- (1,-3.5);
\draw[ very thick] (2.2,0) -- (2.2,-3.5);
 
\draw[very thick] (2.7,-1.8) .. controls +(0,-1) and +(0,1) .. (6.4,-3.5);

\draw[white,   line width=6pt] (3.2,0) -- (3.2,-3.5);
\draw[ very thick] (3.2,0) -- (3.2,-3.5);
\draw[very thick] (6.4,0) .. controls +(0,-1) and +(0,1) .. (2.7,-1.8);

\draw[ very thick] (4.2,0) -- (4.2,-3.5);
\draw[ very thick] (5.4,0) -- (5.4,-3.5);
\draw[ very thick] (7.4,0) -- (7.4,-3.5);
\draw[ very thick] (8.6,0) -- (8.6,-3.5);

\draw  (5.4,-0.7) circle (0.13cm);
\draw  (4.2,-0.97) circle (0.13cm);
\draw  (5.4,-2.82) circle (0.13cm);
\draw  (4.2,-2.57) circle (0.13cm);
\draw  (3.2,-1.27) circle (0.13cm);
\draw (4.8, -4) node {$x_{ij}$};
\end{tikzpicture}
\hspace{1.6cm}
\begin{tikzpicture}[scale=0.6]
\draw (1, 0.3) node {$_{1}$};
\draw (2.2, 0.3) node {$_{i-1}$};
\draw (3.2, 0.3) node {$_{i}$};
\draw (4.2, 0.3) node {$_{i+1}$};
\draw (5.4, 0.3) node {$_{j-1}$};
\draw (6.4, 0.3) node {$_{j}$};
\draw (7.4, 0.3) node {$_{j+1}$};
\draw (8.6, 0.3) node {$_{n}$};
\draw (1.6, -1.8) node {$\cdots$};
\draw (8, -1.8) node {$\cdots$};
\draw (4.8, -1.8) node {$\cdots$};

\draw[very thick] (6.4,0) .. controls +(0,-1) and +(0,1) .. (2.7,-1.8); 
\draw[ very thick] (1,0) -- (1,-3.5);
\draw[ very thick] (2.2,0) -- (2.2,-3.5);

\draw[white,   line width=6pt] (3.2,0) -- (3.2,-3.5);
\draw[ very thick] (3.2,0) -- (3.2,-3.5);
\draw[very thick] (2.7,-1.8) .. controls +(0,-1) and +(0,1) .. (6.4,-3.5);

\draw[ very thick] (4.2,0) -- (4.2,-3.5);
\draw[ very thick] (5.4,0) -- (5.4,-3.5);
\draw[ very thick] (7.4,0) -- (7.4,-3.5);
\draw[ very thick] (8.6,0) -- (8.6,-3.5);

\draw  (5.4,-0.7) circle (0.13cm);
\draw  (4.2,-0.97) circle (0.13cm);
\draw  (5.4,-2.82) circle (0.13cm);
\draw  (4.2,-2.57) circle (0.13cm);
\draw  (3.2,-2.28) circle (0.13cm);
\draw (4.8, -4) node {$x_{ji}$};
\end{tikzpicture}
\caption{ The virtual pure braids $x_{ij}$ and $x_{ji}$ for $i<j$. \label{fig:purebraids}}
\end{figure} 

The {\em upper-triangular pure virtual braid group}\/, $\PV_n^+$, 
is the subgroup of $\PV_n$ generated by those elements $x_{ij}$ 
with $1\leq i<j \leq n$.  Its defining relations are  
\begin{align} 
&x_{ij}x_{ik}x_{jk}=x_{jk}x_{ik}x_{ij} &&\text{for $i<j<k$}, \\ \notag
&x_{ij}x_{kl}=x_{kl}x_{ij} &&\text{for $i\neq j\neq k\neq l$, $i<j$, and $k<l$}.
\end{align}

Of course, both $\PV_1$ and $\PV_1^+$ are the trivial group. 
It is readily seen that $\PV_2^+=\Z$, while 
$\PV_3^+\cong \Z\ast \Z^2$.  Likewise, $\PV_2$ is 
isomorphic to $F_2$.

\subsection{Split monomorphisms}

In \cite{Bartholdi-E-E-R}, the group $\PV_n$ is called the {\em quasi-triangular 
group}, and is denoted by $\QTr_n$, while the quotient group of $\QTr_n$ by the 
relations of the form $x_{ij}=x_{ji}$ for $i\neq j$ is called the {\em triangular 
group}, and is denoted by $\Tr_n$.

\begin{lemma}
\label{lem:iso}
The group $\Tr_n$ is isomorphic to $\PV_n^+$.
\end{lemma}

\begin{proof}
Let $\phi\colon \Tr_n\rightarrow \PV_n^+$ be the homomorphism defined 
by $\phi(x_{ij})=x_{ij}$ for $i<j$ and $\phi(x_{ij})=x_{ji}$ for $i>j$, 
and let $\psi\colon \PV_n^+ \rightarrow \Tr_n$  be the homomorphism 
defined  by $\psi(x_{ij})=x_{ij}$ for $i<j$. It is easy to show that $\phi$ 
and $\psi$ are well-defined homomorphisms, 
and $\phi\circ\psi= {\rm id}$ and $\psi\circ\phi= {\rm id}$. 
Thus, $\phi$ is an isomorphism.
\end{proof}

\begin{corollary}
\label{cor:split}
The inclusion $j_n\colon \PV_n^+\inj \PV_{n}$ is a split monomorphism. 
\end{corollary}

\begin{proof}
The split surjection is defined by the composition of the quotient surjection 
$\PV_n \surj \Tr_n$ and the map $\phi\colon \Tr_n\to \PV_n^+$ from Lemma \ref{lem:iso}.
\end{proof}
 
There are several other split monomorphisms between the aforementioned 
groups. 

\begin{lemma}
\label{lem:split}
For each $n\ge 2$, there are split monomorphisms 
$\iota_n\colon \PV_n\rightarrow \PV_{n+1}$ and
$\iota_n^+\colon \PV_n^+\rightarrow \PV_{n+1}^+$.
\end{lemma}

\begin{proof}
The maps $\iota_n$ and $\iota_n^+$ are defined by sending the generators of 
$\PV_n$ to the generators of $\PV_{n+1}$ with the same indices. 
The split surjection $\pi_n \colon \PV_{n+1}\surj \PV_n$ sends $x_{ij}$ 
to zero for $i=n+1$ or $j=n+1$, and sends $x_{ij}$ to $x_{ij}$ otherwise. 
The split surjection $\pi_n^+ \colon \PV_{n+1}^+ \surj \PV_n^+$ is 
defined similarly. 
\end{proof}

As noted by Bardakov in \cite[Lemma 6]{Bardakov04}, the pure virtual braid 
group $\PV_n$ admits a semi-direct product decomposition of the form  
$\PV_n\cong F_{q(n)}\rtimes \PV_{n-1}$, where $q(1)=2$ and 
$q(n)$ is infinite for $n\geq 2$. Furthermore, as shown in \cite{Bartholdi-E-E-R}, 
there exists a monomorphism from $P_n$ to $\vP_n$.

\subsection{A free product decomposition for $\PV_3$}
\label{subsec:pv23}
The pure virtual braid group $\PV_3$ is generated by 
$x_{12}, x_{21},x_{13}$, $x_{31},x_{23},x_{32}$, subject to the relations
\begin{align*}
& x_{12} x_{13}x_{23} = x_{23}x_{13} x_{12},
&x_{21} x_{23}x_{13} = x_{13} x_{23} x_{21}, 
&& x_{13} x_{12} x_{32}=x_{32}x_{12}x_{13},\\ 
&x_{31} x_{32}x_{12}   = x_{12}x_{32} x_{31},
&x_{23} x_{21} x_{31} =x_{31}x_{21}x_{23},
&& x_{32} x_{31}x_{21}=x_{21}x_{31}x_{32}.
\end{align*}

The next lemma gives a free product decomposition for this group,  
which will play an important role in the proof that $\PV_3$ is $1$-formal.

\begin{lemma}
\label{lem:isogroups}
There is a free product decomposition  $\PV_3\cong \overline{P}_4\ast \Z$.
\end{lemma} 

\begin{proof}
As shown in \cite{Bardakov-M-V-W}, if we set $a_1 =
x_{13} x_{23}$, $b_1 =
x_{13} x_{12}$,
$b_2 =x_{21} x_{31}$, $a_2 =
x_{32} x_{31},
c_1 = x_{13} x_{31}$, $c_2 =
x_{13}$, then there is a free product decomposition,  
$\PV_3\cong G_3\ast \Z$, where $G_3$ is generated by 
$\{a_1, a_2, b_1, b_2, c_1\}$, subject to the 
relations
\[
[a_1, b_1] = [a_2, b_2] = 1,\quad 
b_1^{c_1} = b_1^{a_2}, \quad 
a_1^{c_1} = a_1^{b_2}, \quad 
b_2^{c_1} = b_2^{a_1 b_2}, \quad 
a_2^{c_1} = a_2^{b_1 a_2}, 
\]
where $y^x =  x^{-1} y x$. Replacing the generators in the presentation 
of $G_3$ by $x_1$, $x_2$, $x_3$, $x_4$, $x_5^{-1}$, respectively, 
and simplifying the relations, we obtain a new presentation 
for the group $G_3$, with generators $x_1,\dots,x_5$ and relations
\[
x_1 x_3=x_3x_1,\  
x_2 x_4=x_4x_2,\  
x_5x_3x_2=x_3x_2x_5=x_2x_5x_3,\
x_1x_4x_5=x_4x_5x_1=x_5x_1x_4.
\]

On the other hand, as noted for instance in \cite{Cohen-Suciu95}, 
the group $\overline{P}_4$ has a presentation with generators 
$z_1,\dots ,z_5$ and relations
\[
z_2 z_3=z_3z_2,\:
z_2^{-1}z_4z_2z_1=z_1z_2^{-1}z_4z_2,\:   
z_5z_3z_1=z_3z_1z_5=z_1z_5z_3,\:
z_5z_4z_2=z_4z_2z_5=z_2z_5z_4.
\]
Define a homomorphism $\phi\colon G_3\to \overline{P}_4$ 
by sending $x_1\mapsto z_2$, $x_2\mapsto z_1$, $x_3\mapsto z_3$, 
$x_4\mapsto z_2^{-1}z_4z_2$ and $x_5\mapsto z_5$. A routine check 
shows that $\phi$ is a well-defined homomorphism, with  inverse 
$\psi\colon \overline{P}_4\to G_3$ sending 
$z_1\mapsto x_2$, $z_2\mapsto x_1$, $z_3 \mapsto x_3$, 
$z_4 \mapsto x_1x_4x_1^{-1}$, and $z_5\mapsto x_5$. 
This completes the proof.
\end{proof} 

As a quick application of this lemma,
we obtain the following corollary, which was first proved
by Bardakov et al.~\cite{Bardakov-M-V-W} using a different method.

\begin{corollary}
\label{cor:rtfn}
The pure virtual braid group $\vP_3$ is a residually torsion-free nilpotent group.
\end{corollary}

\begin{proof}
It is readily seen that two groups $G_1$ and $G_2$ are residually 
torsion-free nilpotent if and only their direct product, $G_1\times G_2$, 
is residually torsion-free nilpotent.  Now, as shown by Falk and Randell  
in \cite{Falk-Randell88}, the pure braid groups $P_n$ are 
residually torsion-free nilpotent.  Hence, the subgroup 
$\overline{P}_4\subset P_4$ is also residually torsion-free nilpotent.   

On the other hand, Malcev \cite{Malcev49} showed that if $G_1$ and $G_2$ 
are residually torsion-free nilpotent groups, then the free product $G_1*G_2$ is also 
residually torsion-free nilpotent. The claim follows from the decomposition 
$\vP_3\cong \overline{P}_4*\Z$.
\end{proof}
 
A more general question was asked by Bardakov and Bellingeri in \cite{Bardakov-B09}:
Are the groups $vP_n$ or $\vP_n^+$ residually torsion-free nilpotent? 

\section{Cohomology rings and Hilbert series}
\label{sec:coho}

In this section we discuss what is known about the cohomology rings 
of the pure (virtual) braid groups, and the corresponding Hilbert series.

\subsection{Hilbert series and generating functions}
\label{subsec:hilb}

Recall that the (ordinary) {\em generating function}\/ for a 
sequence of power series $\mathbf{P}=\{p_n(t)\}_{n\geq 1}$ is defined by 
$F(u,t):=\sum_{n=0}^{\infty}p_n(t)u^n$.
Likewise, the {\em exponential generating function}\/ 
for $\mathbf{P}$ is defined by 
$E(u,t):=\sum_{n=0}^{\infty}p_n(t)\frac{u^n}{n!}$.

Now let $\mathbf{G}=\{G_n \}_{n\geq 1}$ be a sequence of groups 
admitting classifying spaces $K(G_n,1)$ with finitely many cells in each dimension.  
We then define the exponential generating function for the 
corresponding Poincar\'{e} polynomials by 
\begin{equation}
\label{eq:poinseries}
\Poin(\mathbf{G},u,t):=1+\sum_{n=1}^{\infty}\Poin(G_n,t)\frac{u^n}{n!}.
\end{equation}
In particular, if we set $t=-1$, we obtain the exponential generating function for 
the Euler characteristics of the groups $G_n$, denoted by $\chi(\bG)$.

For instance, the Poincar\'{e} polynomial of a free group of rank $n$ 
is $\Poin(F_n,t)=1+nt$.  Thus, the exponential generating function for the 
sequence $\mathbf{F}=\{F_n \}_{n\geq 1}$ is $\Poin(\mathbf{F},u,t)=(1+tu)e^u$.

\subsection{Pure braid groups}
\label{subsec:coho pn}
A classifying space for the pure braid group $P_n$ is the configuration space 
$\Conf(\C,n)$ of $n$ distinct points on the complex line. This 
space has the homotopy type of a finite, $(n-1)$-dimensional CW-complex.  
As shown by Arnold in \cite{Arnold69}, 
the cohomology algebra $A_n=H^*(P_n;\C)$ is the 
skew-commutative ring generated by degree $1$ elements 
$a_{ij}$ ($1\leq i<j\leq n$), subject to the relations
\begin{equation}
\label{eq:arnold}
a_{ik}a_{jk}=a_{ij}(a_{jk}-a_{ik}) \:\text{ for }\: i<j<k.
\end{equation}

Clearly, this algebra is quadratic.    
In fact, $A_n$ is a Koszul algebra, that is to say, $\Ext^i_{A_n}(\C,\C)_j=0$ for $i\ne j$.  
Furthermore, the Poincar\'{e} polynomial of $\Conf(\C,n)$, or, equivalently, the Hilbert 
series of $A_n$, is given by 
\begin{equation}
\label{eq:hilbseriesP}
\Poin(P_n,t) = \prod_{k=1}^{n-1}(1+kt)=\sum_{i=0}^{n-1}c(n,n-i)\, t^i,
\end{equation}
where $c(n,m)$ are the (unsigned) Stirling numbers of the first kind, 
counting the number of permutations of $n$ elements which 
contain exactly $m$ permutation cycles. 

\begin{prop}
\label{prop:genf pn}
The exponential generating function for the Poincar\'{e} polynomials 
of the pure braid groups $P_n$ is given by
\[
\Poin(\mathbf{P}, u,t)=\exp\left(-\frac{\log(1-tu)}{t}\right).
\]
\end{prop}
\begin{proof}
It is known (see, e.g.~\cite{Stanley}) that the exponential generating 
function for the unsigned Stirling numbers $c(n,k)$ is given by
\begin{equation}
\label{eq:stir}
\exp(-x\cdot {\log(1-z)})=\sum_{n=0}^{\infty}\sum_{k=0}^{n}c(n,k)x^k\frac{z^n}{n!}.
\end{equation}

Setting $x=t^{-1}$ and $z=tu$, we obtain
\begin{equation}
\exp\left(-\frac{\log(1-tu)}{t}\right)=
\sum_{n=0}^{\infty}\sum_{k=0}^{n}c(n,k)t^{-k}\frac{(tu)^n}{n!}=
1+\sum_{n=1}^{\infty}\sum_{i=0}^{n-1}c(n,n-i)t^i \frac{u^n}{n!},
\end{equation}
where we used $c(0,0)=1$ and $c(n,0)=0$ for $n\geq 1$. This completes the proof.
\end{proof}

\subsection{Pure virtual braid groups}
\label{subsec:coho}
In \cite{Bartholdi-E-E-R}, Bartholdi et al.~describe  
classifying spaces for the pure virtual braid group $\vP_n$ and $\vP_n^+$. 
Let us note here that both these spaces are finite, $(n-1)$-dimensional 
CW-complexes.

The following theorem provides presentations for the cohomology algebras  
of the pure virtual braid groups and their upper triangular subgroups.

\begin{theorem}[\cite{Bartholdi-E-E-R, Lee}]
\label{thm:BEER}
For each $n\ge 2$, the following hold.
\begin{enumerate}
\item \label{an}
The cohomology algebra $A_n=H^*(\PV_n;\C)$ is the skew-commutative algebra 
 generated by degree $1$ elements $a_{ij}$ $(1\leq i\neq j\leq n)$  
subject to the relations   
$a_{ij}a_{ik} =a_{ij}a_{jk}-a_{ik}a_{kj}$, 
$a_{ik}a_{jk} =a_{ij}a_{jk}-a_{ji}a_{ik}$, and 
$a_{ij}a_{ji}=0$ for $i, j, k$ all distinct.
\item \label{an+}
The cohomology algebra $A^+_n=H^*(\PV^+_n;\C)$ is the skew-commutative algebra
generated by degree $1$ elements $a_{ij}$ $(1\leq i\neq j\leq n)$,   
subject to the relations $a_{ij}=-a_{ji}$ and  $a_{ij}a_{jk}=a_{jk}a_{ki}$ 
for $i\neq j\neq k$.
\end{enumerate}
\end{theorem}

\begin{corollary}
\label{cor:presentationCoho}
The cohomology algebra $H^*(\PV^+_n;\C)$ has a simplified presentation 
with generators $e_{ij}$ in degree $1$ for $1\leq i< j\leq n$, 
and relations $e_{ij}(e_{ik}-e_{jk})$ and  $(e_{ij}-e_{ik})e_{jk}$ for $i< j< k$.
\end{corollary}

\begin{proof}
Let $\widetilde{A}_n^+$ be the algebra given by the above presentation.  
The morphism $\phi\colon A_n^+\to \widetilde{A}_n^+$ defined by 
$\phi(a_{ij})=e_{ij}$ for $i<j$ and $\phi(a_{ij})=-e_{ij}$ for $i>j$ is easily 
checked to be an isomorphism.
\end{proof}

In \cite{Bartholdi-E-E-R}, Bartholdi et al.~also showed that both $A_n$ 
and $A^+_n$ are Koszul algebras, and computed the Hilbert series of 
these graded algebras, as follows:
\begin{equation}
\label{eq:hilbseriesvP}
\Poin(vP_n, t) =\sum_{i=0}^{n-1}L(n,n-i)\, t^i , \qquad 
\Poin(vP_n^+, t) =\sum_{i=0}^{n-1}S(n,n-i)\, t^i.
\end{equation}
Here $L(n,n-i)$ are the Lah numbers, i.e.,
the number of ways of partitioning $[n]$ into $n-i$ nonempty ordered subsets,
while $S(n,n-i)$ are the Stirling numbers of the second kind, 
i.e., the number of ways of partitioning $[n]$ into $n-i$ nonempty (unordered) sets. 
Explicitly,
\begin{align}
\label{eq:lah}
L(n,n-i)&= \binom{n-1}{i} \dfrac{n!}{(n-i)!}, \\ 
\label{eq:stirling2}
S(n,n-i)&= \dfrac{1}{(n-i)!}\sum\limits_{j=0}^{n-i-1}(-1)^j
\binom{n-i}{j}(n-i-j)^n.
\end{align}
The polynomial $\Poin(vP_n^+, t)$ is the rank-generating 
function for the partition lattice $\Pi_n$, 
see e.g.~\cite[Exercise 3.10.4]{Stanley}.
All the roots of such a polynomial are negative real numbers.  

\begin{prop}[\cite{Bartholdi-E-E-R}]
\label{prop:expgen vpn}
The exponential generating function 
for the polynomials $\Poin(vP_n,t)$ and $\Poin(vP_n^+,t)$ are 
given by
\[
\Poin(\mathbf{vP}, u,t)=\exp\!\left(\frac{u}{1-tu}\right), \qquad 
\Poin(\mathbf{vP^+}, u,t)=\exp\!\left(\frac{\exp(tu)-1}{t}\right).
\]
\end{prop}

Finally, let us note that Dies and Nicas \cite{Dies-Nicas14} 
showed that the Euler characteristic
of $vP_n$ is non-zero for all $n\geq 2$, while the Euler characteristic
of $vP_n^+$ is non-zero for all $n\geq 3$, with one possible exception 
(and no exception if Wilf's conjecture is true). 

\section{Resonance varieties}
\label{sec:res}

In this section we study the resonance varieties of the pure (virtual) braid groups. 
We start with a discussion of how resonance behaves with respect to products and 
wedges. 

\subsection{Resonance varieties of graded algebras}
\label{sub:resonance alg}

Let $A=\bigoplus_{i\ge 0} A^i$ be a graded, graded-commu\-tative $\C$-algebra.  
We will assume throughout that $A$ is connected (i.e., $A^0=\C$), 
and locally finite (i.e., the Betti numbers $b_i:=\dim A^i$ are finite, 
for all $i\ge 0$). By definition, the (degree $i$, depth $d$) 
\emph{resonance varieties}\/ of $A$ 
are the algebraic sets
\begin{equation}
\label{eq:resvars}
\cR^i_d(A)=\{a\in A^1\mid b_i(A,a) \geq d\}, 
\end{equation}
where $(A,a)$ is the cochain complex (known as the {\em Aomoto complex}) 
with differentials $\delta^i_{a}\colon A^i\to A^{i+1}$ given by 
$\delta^i_{a}=a\cdot u$, and $b_i(A,a):=\dim H^i(A,a)$.  

Observe that $b_i(A,0)=b_i(A)$.  Thus, $\cR_d^i(A)$ is empty if either $d>b_i$ 
or $d\geq 0$ and $b_i=0$. Furthermore, $0\in \cR_d^i(A)$ if and 
only if $d\le b_i$. In degree zero, we have that $\cR_d^0(A)=\{0\}$ 
for $d=1$  and $\cR_d^0(A)=\emptyset$ for $d\geq 2$.
We use the convention that $\cR^i_d(A)=A^1$ for $d\leq 0$.  
The following simple lemma will be useful in computing the 
resonance varieties of the algebra $A=H^*(\vP_3,\C)$.

\begin{lemma}
\label{lem:resonance2}
Suppose $A^i\neq 0$ for $i\leq 2$ and $A^i=0$ for $i\geq 3$.  
Then $\cR_{d}^2(A)=\cR_{d-\chi}^1(A)$ for $d\leq b_2$, where 
$\chi=1-b_1+b_2$ is the Euler characteristic of $A$.
\end{lemma}

\begin{proof}
By the above discussion,  $0\in \cR_{d}^2(A)$ if and only if 
$d\leq b_2$.  But this is equivalent to 
$0\in \cR_{d-\chi}^1(A)$, since $d-\chi\leq b_2-\chi\leq b_1-1$.
Now let $a\in A^1\setminus \{0\}$.  Then 
$b_2(A,a) =b_1(A,a)+\chi$. Hence,
$a\in \cR_{d}^2(A)$ if and only if $a\in \cR_{d-\chi}^1(A)$, 
and we are done.
\end{proof}

We will be mostly interested here in the degree $1$ resonance varieties, 
$\cR^1_d(A)$.  Equations for these varieties can be obtained as follows 
(see for instance \cite{Suciu12}). 
Let $\{e_1,\dots,e_n\}$ be a basis for the complex vector space 
$A^1=H^1(G;\C)$, and let $\{x_1,\dots,x_n\}$ be the dual basis for 
$A_1=H_1(G;\C)$. Identifying the symmetric algebra $\Sym(A_1)$ 
with the polynomial ring $S=\C[x_1,\dots,x_n]$, we obtain 
a cochain complex of free $S$-modules,
\begin{equation}
\label{eq:cc}
\xymatrix{A^0\otimes_{\C} S \ar^{\delta^0}[r]& A^1\otimes_{\C} 
S \ar^{\delta^1}[r]& A^2\otimes_{\C} S \ar^(.6){\delta^2}[r] &  \cdots},
\end{equation}
with differentials given by
$\delta^i(u\otimes 1)=\sum_{j=1}^ne_ju\otimes x_j$ 
for $u\in A^i$ and extended by $S$-linearity.  
The resonance variety $\cR^1_d(A)$, then, is the zero locus of the 
ideal of codimension $d$ minors of the matrix $\delta^1$.

Now suppose $X$ is a connected, finite-type CW-complex. 
One defines then the resonance varieties of $X$ to be the sets 
$\cR^i_d(X):=\cR^i_d(H^*(X,\C))$.  Likewise, the resonance varieties
of a group $G$ admitting a finite-type classifying space are defined as 
$\cR^i_d(G):=\cR^i_d(H^{*}(G,\C))$.

\subsection{Resonance varieties of products and coproducts}
\label{sub:resonance coprod}

The next two results are generalizations of Propositions 13.1 and 13.3 
from \cite{Papadima-Suciu10P}. We will use these results 
to compute the resonance varieties of the group $\vP_3$.

\begin{prop}
\label{prop:resProd}
Let $A=B \otimes C$ be the product of two connected, finite-type 
graded algebras. Then, 
for all $ i\ge 1$, 
\begin{align*}
\cR^1_d(B \otimes C)&=\cR^1_d(B)\times \{0\} \cup \{0\}\times \cR^1_d(C),\\
\cR^i_1(B \otimes C)&=\bigcup\limits_{s+t=i} \cR^s_1(B)\times  \cR^t_1(C).
\end{align*}
\end{prop}

\begin{proof}
Let $a = (a_1, a_2)$ be an element in $A^1=B^1\oplus C^1$. 
The cochain complex $(A,a)$ splits as a tensor product of cochain 
complexes, $(B,a_1)\otimes (C,a_2)$. Therefore, 
\begin{equation}
\label{eq:bettiAomoto}
b_i(A,a)=\sum_{s+t=i} b_s(B,a_1)b_t(C,a_2),
\end{equation}
and the second formula follows.
In particular, we have
$b_1(A,(0,0))=b_1(B, 0)+b_1(C, 0)$, $b_1(A,(0,a_2))=b_1(C, a_2)$ if $a_2\neq 0$, 
$b_1(A,(a_1,0))=b_1(B, a_1)$ if $a_1\neq 0$, and  $b_1(A,a)=0$ if $a_1\neq 0$ and 
$a_2\neq 0$.
The first formula now easily follows.
\end{proof}

\begin{prop}
\label{prop:res}
Let $A=B \vee C$ be the coproduct of two connected, finite-type 
graded algebras. Then, 
for all $ i\ge 1$, 
\begin{align*}
& \cR_d^1(B \vee  C)=
 \bigg(\bigcup\limits_{j+k=d-1}(\cR_j^1(B)\backslash \{0\}) \times 
(\cR_k^1(C)\backslash \{0\})\bigg)\cup  \\[-6pt]
&\hspace*{2in}\big( \{0\}\times \cR^1_{s}(C)\big)\cup 
\big( \cR^1_{t}(B)\times\{0\}\big), \\  
&\cR^i_d(B\vee C)=\bigcup\limits_{j+k=d} \cR^i_j(B)\times  \cR^i_k(C),
 \quad \textrm{ if } i\geq 2,
\end{align*}
where $s=d-\dim B^1$ and $t=d-\dim C^1$.
\end{prop}

\begin{proof}
Pick an element $a=(a_1, a_2)$ in $A^1=B^1\oplus C^1$.  The Aomoto complex of $A$ 
splits (in positive degrees) as a direct sum of chain complexes, 
$(A^{+},a) \cong (B^{+},a_1) \oplus (C^{+},a_2)$. 
We then have formulas relating the Betti numbers of the respective 
Aomoto complexes:
\[
b_i(A,a)=\begin{cases}
b_i(B, a_1)+b_i(C, a_2)+1&\text{if $i=1$, 
and $a_1\ne 0$, $a_2\ne 0$,}\\
b_i(B, a_1)+b_i(C, a_2)&\text{otherwise.}
\end{cases}
\]
The claim follows by a case-by-case analysis of the 
above formula.
\end{proof}

\subsection{Pure braid groups}
\label{subsec:res pn}
Since $P_n$ admits a classifying space of dimension $n-1$, 
the resonance varieties $\cR_d^i(P_n)$ are empty for $i\geq n$.
In degree $i=1$,  the resonance varieties $\cR_d^i(P_n)$
are either trivial, or a union of $2$-dimensional subspaces.  

\begin{prop} [\cite{Cohen-Suciu99}]
\label{prop:res pn}
The first resonance variety of the pure braid group $P_n$
has decomposition into irreducible components given by
\begin{equation*}
 \cR_1^1(P_n)=    \bigcup_{1\leq i<j<k\leq n} L_{ijk}\cup 
 \bigcup_{1\leq i<j<k<l\leq n} L_{ijkl}, 
\end{equation*}
where 
\begin{align*}
&L_{ijk}=\big\{x_{ij}+x_{ik}+x_{jk}=0 \text{ and } x_{st}=0 
\textrm{~if~} \{s,t\}\not\subset \{i,j,k\}  \big\}, \\
&L_{ijkl}=\left\{
\begin{array}{l}
\sum_{\{p,q\}\subset \{i,j,k,l\}}x_{pq}=0,\:
x_{ij}=x_{kl},\: x_{jk}=x_{il},\: x_{ik}=x_{jl},\\[2pt]
x_{st}=0 \textrm{~if~}\{s,t\}\not\subset \{i,j,k,l\}
\end{array}
\right\}. 
\end{align*}
Furthermore, $\cR_d^1(P_n)=\{0\}$ for $2\le d\le \binom{n}{2}$, 
and  $\cR_d^1(P_n)=\emptyset$ for $d> \binom{n}{2}$.
\end{prop}

Recall that $P_n\cong \overline{P}_n\times \Z$, where $\overline{P}_n$ 
is the quotient of $P_n$ by its (infinite cyclic) center. 
Thus, the resonance varieties of the group $\overline{P}_n$ can be 
described in a similar manner, using Proposition \ref{prop:resProd}. 

\subsection{Resonance varieties of $\vP_3$}
\label{subsec:res vp3}

A partial computation of the resonance varieties $\cR_d^i(\vP_3)$
was done in a preprint version of \cite{Bardakov-M-V-W} for $i=1$ 
and $d=1, 5, 6$, as well as $i=2$ and $d=2, 6$.  
We use the preceding discussion to give a 
complete computation of all these varieties.

\begin{prop}
\label{prop:vP3}
For $d\geq 1$, the resonance varieties of the pure virtual braid group $\vP_3$ 
are given by
\begin{equation*}
\cR_d^i(\vP_3)\cong
\begin{cases}
  \cR_{d-1}^1(\overline{P}_4)\times \C,  & \textrm{ for } i=1, d\leq 5\\  
    \{0\} &\textrm{ for } i=1, d=6,\\  
 \cR_{d-2}^1(\overline{P}_4)\times \C   & \textrm{ for } i=2, d\leq 6 \\  
 \emptyset & otherwise.
\end{cases}
\end{equation*}
Consequently, $\cR_1^1(\vP_3)=\C^6$, while $\cR_2^1(\vP_3)$ 
is a union of five $3$-dimensional subspaces, pairwise intersecting 
in the $1$-dimensional subspace 
$\cR_3^1(\vP_3)=\cR_4^1(\vP_3)=\cR_5^1(\vP_3)$.
\end{prop}
 
\begin{proof}
By Lemma \ref{lem:isogroups}, we have an isomorphism $\vP_3\cong\overline{P}_4*\Z$,
which yields an isomorphism $H^1(\vP_3;\C)\cong H^1(\overline{P}_4;\C)\oplus \C$.
Under this identification, Proposition \ref{prop:res} shows that 
$\cR^1_d(\vP_3)\cong \cR^1_{d-1}(\overline{P}_4)\times \C$ for $d\leq 5$, and 
$\cR^1_6(\vP_3)=\{0\}$. 

The same proposition also shows that  
$\cR^2_d(\vP_3)\cong \cR^2_{d}(\overline{P}_4)\times \C$. 
On the other hand, by Lemma \ref{lem:resonance2},
we have that $\cR_d^2(\overline{P}_4)=\cR_{d-2}^1(\overline{P}_4)$ 
for $d\leq 6$, since $\overline{P}_4$ admits a $2$-dimensional classifying 
space, and $\chi(\overline{P}_4)=2$.
Finally, the description of the resonance varieties $\cR^1_d(\vP_3)$ for $d\le 6$ 
follows from  Proposition \ref{prop:res pn}.   
 \end{proof}

Let $a_{12}, a_{13}, a_{23},a_{21}, a_{31}, a_{32}$ be 
the basis of $H^1(\PV_3,\C)$ specified in Theorem \ref{thm:BEER}, 
and let $x_{ij}$ the corresponding coordinate functions on this 
affine space. 
Tracing through the isomorphisms  
$H^1(\PV_3,\C)\cong H^1(\overline{P}_4,\C)\times \C \cong 
H^1(P_4,\C)$, we see that the components of $\cR^1_2(\PV_3)$ 
have equations 
\begin{align*}
&\{x_{12}-x_{23}=x_{12}+x_{32}=x_{12}+x_{21}=0\}, 
&\{x_{13}+x_{23}=x_{12}+x_{32}=x_{21}+x_{31}=0\}, \\[-2pt]
&\{x_{13}+x_{23}=x_{13}-x_{32}=x_{13}+x_{31}=0\}, 
&\{x_{12}+x_{13}=x_{12}+x_{21}=x_{12}-x_{31}=0\}, \\[-2pt]
&\{x_{12}+x_{13}=x_{23}+x_{21}=x_{31}+x_{32}=0\}, 
\end{align*} 
while their common intersection is the line 
$\{x_{12}=-x_{21}=-x_{13}=x_{31}=x_{23}=-x_{32}\}$.

\subsection{Resonance varieties of $\PV^+_4$}
\label{subsec:res vp4plus}

We now switch to the upper-triangular group $\PV^+_4$, and 
compute its degree $1$ resonance varieties.   Let 
$e_{12}, e_{13}, e_{23},e_{14}, e_{24}, e_{34}$ be 
the basis of $H^1(\PV_4^+,\C)$ specified in Corollary \ref{cor:presentationCoho}, 
and let $x_{ij}$ be the corresponding coordinate functions on this 
affine space. 

\begin{lemma}
\label{lem:resonancePV4+}
The depth $1$ resonance variety $\cR_1^1(\PV_4^+)$ is the irreducible, 
$4$-dimensional subvariety of degree $6$ inside 
$H^1(\PV_4^+,\C)=\C^6$ defined by the equations
\[
\left\{
\begin{array}{ll}
x_{12}x_{24}(x_{13}+x_{23})+
x_{13}x_{34}(x_{12}-x_{23})-x_{24}x_{34}(x_{12}+x_{13})=0,\\
x_{12}x_{23}(x_{14}+x_{24})+x_{12}x_{34}(x_{23}-x_{14})+
x_{14}x_{34}(x_{23}+x_{24})=0,\\
x_{13}x_{23}(x_{14}+x_{24})+x_{14}x_{24}(x_{13}+x_{23}) 
+x_{34}(x_{13}x_{23}-x_{14}x_{24})=0,\\
x_{12}(x_{13}x_{14}-x_{23}x_{24})+x_{34}(x_{13}x_{23}-x_{14}x_{24})=0.
\end{array}\right.
\]
The depth $2$ resonance variety $\cR^1_2(\PV_4^+)$ consists of $13$ lines in $\C^6$, 
spanned by the vectors
\[
\begin{array}{lll}
e_{12},\:\ e_{13},\:\ e_{23},
&e_{14},\:\ e_{24},\:\ e_{34}, 
&e_{12}-e_{13}+e_{23},
\\
e_{12}-e_{14}+e_{24},
&e_{13}-e_{14}+e_{34},
&e_{23}-e_{24}+e_{34},
\\
e_{13}-e_{23}-e_{14}+e_{24},
&e_{12}+e_{23}-e_{14}+e_{34},
&e_{12}-e_{13}+e_{24}-e_{34}. 
\end{array}
\]
Finally, $\cR^1_d(\PV_4^+)=\{0\}$  if $3\le d\le 6$, and $\cR^1_d(\PV_4^+)=\emptyset$  
if $d\ge 7$. 
 \end{lemma}
 
\begin{proof}
Let $A=H^*(\PV^+_4,\C)$ be the cohomology algebra of $\PV^+_4$, as 
described in Corollary \ref{cor:presentationCoho}.  The differential 
$\delta^1\colon A^1\otimes S \to A^2\otimes S$ in the cochain complex 
\eqref{eq:cc} is then given by 
\small
\begin{equation*}
{
\delta^1=\begin{pmatrix}
{-{x}_{34}}&
      0&
      0&
      0&
      0&
      {x}_{12}\\
      -{x}_{13}-{x}_{23}&
      {x}_{12}-{x}_{23}&
      {x}_{12}+{x}_{13}&
      0&
      0&
      0\\
      0&
      {-{x}_{24}}&
      0&
      0&
      {x}_{13}&
      0\\
      0&
      0&
      {-{x}_{14}}&
      {x}_{23}&
      0&
      0\\
      -{x}_{14}-{x}_{24}&
      0&
      0&
      {x}_{12}-{x}_{24}&
      {x}_{12}+{x}_{14}&
      0\\
      0&
      -{x}_{14}-{x}_{34}&
      0&
      {x}_{13}-{x}_{34}&
      0&
      {x}_{13}+{x}_{14}\\
      0&
      0&
      -{x}_{24}-{x}_{34}&
      0&
      {x}_{23}-{x}_{34}&
      {x}_{23}+{x}_{24}
\end{pmatrix} \, .
}
\end{equation*} 
\normalsize
Computing with the aid of Macaulay2 \cite{M2} the elementary ideals of 
this matrix and finding their primary decomposition leads to the stated 
conclusions. 
\end{proof} 
  
The degree $1$, depth $1$ resonance varieties of the virtual pure braid groups 
$\PV^+_n$ and $\PV_n$ can be computed in a similar fashion, at least for small values 
of $n$. For instance, $\cR_{1}^1(vP^+_5)$ has $15$ irreducible components 
of dimension $4$, while $\dim \cR_{1}^1(vP^+_n)=n-1$ for $n=6,7,8$.  
In general, though, these varieties are not equidimensional.  
For example, $\cR_{1}^1(vP_4)$ has seven irreducible components of dimension 
$4$, three irreducible components of dimension $5$, and four irreducible components of 
dimension $6$.  

\section{Graded Lie algebras and graded formality}
\label{sect:liealgebras}

We now discuss the holonomy and associated graded Lie algebras 
of the pure (virtual) braid groups. For relevant background, we refer to \cite{SW1} and 
references therein. 

\subsection{Associated graded Lie algebras}
\label{subsec:grg}

Let $G$  be a finitely generated group. There are several 
Lie algebras associated to such a group.  
The most classical one is the associated graded Lie 
algebra (over $\C$), 
\begin{equation}
\label{eq:grg}
\gr(G):=\bigoplus\limits_{k\geq 1}(\Gamma_kG/\Gamma_{k+1}G)\otimes_{\Z}\C,
\end{equation}
where $\{\Gamma_kG\}_{k\geq 1}$ is the \emph{lower central series}, 
defined inductively by $\Gamma_1G=G$ and
$\Gamma_{k+1}G=[\Gamma_kG,G], k\geq 1$, and the Lie bracket $[x,y]$ is 
induced from the group commutator $[x,y]=xyx^{-1}y^{-1}$. 
We denote by $U=U(\gr(G))$ the universal enveloping algebra 
of $\gr(G)$. By the
Poincar\'{e}--Birkhoff--Witt theorem, we have that
\begin{equation}
\label{eq:Koszuldual}
\prod_{k=1}^{\infty} (1-t^k)^{-\phi_k(G)} =  \Hilb(U,t).
\end{equation}

The next two lemmas give explicit ways to compute the LCS ranks of a group $G$, 
under a common rationality hypothesis for the Hilbert series of the graded Lie algebra 
$U$. 

\begin{lemma}
\label{lem:PBW}
Suppose there is a polynomial $f(t)=1+\sum_{i=1}^n b_it^i\in \Z[t]$ 
such that 
\begin{equation}
\label{eq:ft}
\Hilb(U(\gr(G)),-t)\cdot f(t)=1.
\end{equation}
Then the LCS ranks of $G$ are given by 
\begin{equation}
\label{eq:lcs ranks}
\phi_k(G)=\dfrac{1}{k} \sum_{d|k}\mu\left(\frac{k}{d}\right)\left[
\sum_{m_1+2m_2+\cdots+nm_n=d} (-1)^{s_n}  d(m!) 
\prod_{j=1}^{n} \frac{(b_j)^{m_j}}{ (m_j) !}  \right],
\end{equation}
where $0\leq m_j\in \Z$, $s_n=\sum_{i=1}^{[n/2]}m_{2i}$, 
$m=\sum_{i=1}^{n}m_{i}-1$ and $\mu$ is the M\"{o}bius function.
\end{lemma}
\begin{proof}
From formula \eqref{eq:Koszuldual} and assumption \eqref{eq:ft}, we have that 
\begin{equation}
\label{eq:lcs g}
\prod_{k=1}^{\infty} (1-t^k)^{\phi_k(G)} =  1+\sum_{i=1}^n  b_i(-t)^i.
\end{equation}

Taking logarithms on both sides, we find that 
\begin{equation}
\label{eq:sum pig}
\sum_{j=1}^{\infty}\sum_{s=1}^{\infty}\phi_s(G) \frac{t^{sj}}{j}=
\sum_{w=1}^{\infty}\frac{1}{w} \left(-\sum_{i=1}^n b_i(-t)^{i}\right)^w .
\end{equation}

Comparing  the coefficients of $t^k$ on each side gives 
\begin{equation}
\label{eq:sum dk}
\sum_{d|k} \phi_d(G) \frac{d}{k}=
\sum_{m_1+2m_2+\cdots+nm_n=k} (-1)^{s_n}  (m!) 
\prod_{j=1}^{n} \frac{(b_j)^{m_j}}{ (m_j) !},
\end{equation}
where $s_n=\sum_{i=1}^{[n/2]}m_{2i}$ and $m=\sum_{i=1}^{n}m_{i}-1$.
Finally, multiplying both sides by $k$ and using the M\"{o}bius inversion formula 
yields the desired formula.
\end{proof}

An alternative way of computing the LCS ranks of a group $G$ satisfying 
the assumptions from Lemma \ref{lem:PBW} was given by Weigel in \cite{Weigel15}.
 
\begin{lemma}[\cite{Weigel15}]
\label{lem:Weigel}
Suppose there is a polynomial $f(t)=1+\sum_{i=1}^n b_it^i\in \Z[t]$ 
such that $\Hilb(U(\gr(G),-t))\cdot f(t)=1$.  Let $z_1, \dots , z_n$ be 
the (complex) roots of $f(-t)$.  Then the LCS ranks of $G$ are given by 
\begin{equation}
\label{eq:weigel}
\phi_k(G)=\dfrac{1}{k} \sum_{1\leq i\leq n}\sum_{d|k}
\mu\left(\frac{k}{d}  \right)  \dfrac{1}{z_i^d} .
\end{equation}
\end{lemma}
 
\subsection{Holonomy Lie algebras}
\label{subsec:holo}

Let $A=\bigoplus_{i\ge 0} A_i$ be a graded, graded-commutative, 
algebra over $\C$. We shall assume 
that $A_0=\C$, and $\dim A_1<\infty$. 
Denote by $A_i^*$ the dual vector space to $A_i$. 
We let $\partial_A\colon A_2^*\rightarrow A_1^*\wedge A_1^*$ 
be the dual of the product $A_1\wedge A_1\to A_2$. 
Let  $\bL(A_1^*)$ be the free, complex Lie algebra on $A_1^*$,
with the degree $2$ piece identified with $A_1^*\wedge A_1^*$.
We define the \emph{holonomy Lie algebra}\/ of $A$ as the quotient 
\begin{equation}
\label{eq:holo lie}
\fh(A)=\bL(A_1^*)/\langle\im \partial_A\rangle.
\end{equation}

Now let $G$ be a finitely generated group.  
The holonomy Lie algebra of $G$, denoted by $\fh(G)$, 
is defined to be the holonomy of its cohomology algebra $A=H^*(G;\C)$. 
The identification of the $\C$-vector space $H_{\C}:=H_1(G,\C)$ with 
$\gr_1(G)=G/G'\otimes \C$ 
extends to a surjective morphism of graded Lie algebras, $\bL(H_{\C})\surj \gr(G)$, 
which sends $\im(\partial_G)\subset  \bL^2(H_{\C})$ to zero.  In this 
fashion, we obtain an epimorphism 
\begin{equation}
\label{eq:holonomysurj}
\xymatrix{\Psi\colon \fh(G)\ar@{->>}[r]&  \gr(G)},
\end{equation}
which is readily seen to be an isomorphism in degrees $1$ and $2$.  
Thus, the holonomy Lie algebra $\fh(G)$ can be viewed as the quadratic 
closure of $\gr(G)$. 

The group $G$ is said to be {\em graded-formal}\/ if 
$\Psi$ is, in fact, an isomorphism of graded Lie algebras. 
As noted in \cite{SW1}, the group $G$ is graded-formal if and only if 
$\gr(G)$ is quadratic.
Write 
\begin{equation}
\label{eq:phik}
\phi_k(G)=\dim \gr_k(G) \quad \text{and}\quad  \bar{\phi}_k(G)=\dim \fh_k(G).
\end{equation}

Clearly $\phi_k(G) \le \bar\phi_k(G)$, with equality for $k=1$ and $2$.  
Moreover, the group $G$ is graded-formal if and only if
$\phi_k(G)=\bar{\phi}_k(G)$ for all $k\geq 1$.
 
\begin{prop}
\label{prop:kd}
Suppose the group $G$ is graded-formal, and its cohomology 
algebra, $A=H^*(G;\C)$, is Koszul. Then 
$\Hilb(U(\gr(G)),-t) \cdot \Hilb(A,t)=1$.
\end{prop}

\begin{proof}
Let $U=U(\gr(G))$.   and let $U^{!}$ be its quadratic dual. 
By assumption, $\gr(G)=\h(A)$ is a quadratic Lie algebra. 
Thus, $U$ is a quadratic algebra.  Furthermore, since $A$ is 
also quadratic, $U=U(\h(A))$ is isomorphic to 
$A^{!}$, the quadratic dual of $A$, see \cite{PY}.  

On the other hand, since $A$ is Koszul, 
the Koszul duality formula gives 
$\Hilb(A^{!},-t)\cdot \Hilb(A,t)=1$. The conclusion follows.
\end{proof}

\begin{corollary}
\label{eq:koszul}
Suppose the group $G$ is graded-formal, and its cohomology 
algebra is Koszul and finite-dimensional. Then 
the LCS ranks $\phi_k(G)$ are given by formula \eqref{eq:lcs ranks}, 
where $b_i=b_i(G)$.
\end{corollary}
 
\subsection{Mild presentations}
\label{subsec:mild pres}

Let $G$ be a group admitting a finite presentation 
$F/R=\langle x_1,\dots, x_n\mid r_1,\dots,r_m\rangle$.
The \textit{weight}\/ of a word $r\in F$ is defined as  
$\omega(r)=\sup\{k\mid r\in \Gamma_k F\}$. 
The image of $r$ in $\gr_{\omega(r)}(F)$ is called the 
\textit{initial form}\/ of $r$, and is denoted by $\initial(r)$. 

As before, let $\bL (H_{\C})$ be the free Lie algebra on the 
$\C$-vector space $H_{\C}=H_1(G,\C)$,
and write 
\begin{equation}
\label{eq:flg}
\fL(G):=\bL (H_{\C})/J,
\end{equation}
 where $J$ is the ideal of $\bL (H_{\C})$ 
generated by $\{\initial(r_1),\dots,\initial(r_m) \}$. 
Work of Labute \cite{Labute85} shows that $J/[J, J]$ can be 
viewed as a $U(\fL(G))$-module 
via the adjoint representation of the Lie algebra $\fL(G)$.
If the module $J/[J, J]$ is a free $U(\fL(G))$-module on the 
images of $\initial(r_1),\dots,\initial(r_m)$, 
the given presentation of $G$ are called \emph{mild}.  
If a presentation 
$G=\langle x_1,\dots , x_n\mid  r_1, \dots r_m\rangle$ 
is mild, then the Lie algebra $\gr(G)$ is isomorphic to $\fL(G)$.
As shown by Anick in \cite{Anick87}, the presentation for
$G$ is mild if and only if 
\begin{equation}
\label{eq:anick criterion}
\Hilb(U(\fL(G)),t)= \left(1-nt+\sum_{i=1}^mt^{\omega(r_i)}\right)^{-1}.
\end{equation}

\begin{lemma}
\label{lem:mild}
Let $G$ be a group admitting a mild presentation 
$G=\langle x_1,\dots , x_n\mid  r_1, \dots , r_m\rangle$
such that $r_i\in [F,F]~$ for $1\leq i\leq m$.
If $G$ if graded-formal, then the LCS ranks of 
$G$ are given by
\begin{equation}
\label{eq:lcsMild}
\phi_k(G)=\dfrac{1}{k}\sum_{d|k}\mu(k/d) \frac{\left(n+\sqrt{n^2-4m}\right)^d+ 
\left(n-\sqrt{n^2-4m}\right)^d}{(2m)^d} .
\end{equation}
Moreover, if the enveloping algebra $U=U(\gr(G;\C))$ is Koszul, then 
\[
\Hilb(\Ext_{U}(\C;\C),t)=1+nt+mt^2.
\]
\end{lemma}

\begin{proof}
Since $G$ has a mild presentation, $\gr(G)$ is isomorphic to the Lie algebra $\fL(G)$ 
associated to this presentation. Furthermore, since $G$ is a graded-formal, and all 
the relators $r_i$ are commutators,  we have that 
$\omega(r_i)=2$ for $1\leq i\leq m$.  Using now the Poincar\'{e}--Birkhoff--Witt theorem 
and formula \eqref{eq:anick criterion}, we find that $\Hilb(U(\gr(G)),t)\cdot (1-nt+mt^2)=1$.
Hence, the LCS ranks formula follows from Lemma \ref{lem:Weigel}.

Now suppose $U=U(\gr(G))$ is a Koszul algebra.  Then $\Ext_{U}(\C;\C)=U^!$, and
the expression for the Hilbert series  of $\Ext_{U}(\C;\C)$ follows 
from \eqref{eq:Koszuldual}.
\end{proof}
 
\subsection{Graded algebras associated to $P_n$, $vP_n$ and $\vP_n^+$}
\label{subsec:coho ring vpb}
We start with the pure braid groups $P_n$. 
As shown by Kohno \cite{Kohno85} and Falk--Randell \cite{Falk-Randell85}, 
the graded Lie algebra $\gr(P_n)$ is generated by degree $1$ elements 
$s_{ij}$ for $1\leq i\neq j\leq n$, subjects to the relations 
\begin{equation}
\label{eq:holo pn}
s_{ij}=s_{ji}, \:  [s_{jk},s_{ik}+s_{ij}]=0, \: [s_{ij},s_{kl}]=0 \text{ for $i\neq j\neq l$}. 
\end{equation}
In particular, the pure braid group $P_n$ is graded-formal.
The universal enveloping algebra $U(\gr(P_n))$ is Koszul with 
Hilbert series $\prod_{k=1}^{n-1}(1-kt)^{-1}$.  
Using Koszul duality and formulas \eqref{eq:hilbseriesP} and 
\eqref{eq:lcs ranks}, we see that 
formula \eqref{eq:phi ranks} from the Introduction 
holds for the groups $G_n=P_n$.  Alternatively, the 
ranks $\phi_k(P_n)$ can be computed from formula \eqref{eq:Koszuldual}, 
as follows (see \cite{Kohno85, Falk-Randell85, PY}), 
\begin{equation}
\label{eq:lcsP}
\phi_k(P_n)=\sum_{s=1}^{n-1}\phi_k(F_s)= 
\dfrac{1}{k}\sum_{s=1}^{n-1}\sum_{d|k}\mu(k/d) s^d .
\end{equation}

Next, we give a presentation for the holonomy Lie algebras  of the 
pure virtual braid groups, using a method described in \cite{Papadima-Suciu04,SW1}.   
The Lie algebra $\h(\PV_n)$ is generated by 
$r_{ij}$, $1\leq i\neq j\leq n$, with relations
\begin{equation}
\label{eq:holo pvn}
[r_{ij},r_{ik}]+[r_{ij},r_{jk}]+[r_{ik},r_{jk}]=0 \textrm{ for distinct $i,j,k$},
\end{equation}
and $[r_{ij},r_{kl}]=0$ for distinct $i,j,k,l.$  
The Lie algebra $\h(\PV_n^+)$ is the quotient 
Lie algebra of $\h(\PV_n)$ by the ideal generated by
$r_{ij}+r_{ji}$ for distinct $i\neq j$. Similarly as Corollary \ref{cor:presentationCoho},
the Lie algebra $\h(\PV_n^+)$ has a simplified presentation with generators
$r_{ij}$ for $i<j$ and the corresponding relations of $\fh(vP_n)$. 

\begin{theorem}[\cite{Bartholdi-E-E-R, Lee}]
\label{thm:beer lee}
The Lie algebra $\h(\PV_n)$ is isomorphic to the Lie algebra 
$\gr(\PV_n)$.  Likewise, the Lie algebra $\h(\PV_n^+)$ is 
isomorphic to the Lie algebra $\gr(\PV_n^+)$.
\end{theorem}

The next corollary shows that formula \eqref{eq:phi ranks} from the Introduction 
also holds for the virtual pure braid groups.

\begin{corollary}
\label{cor:lcsranks}
The LCS ranks of the groups $G_n=\vP_n$ 
and $\vP_n^+$ are given by 
\[
\phi_k(G_n)=\dfrac{1}{k} \sum_{d|k}\mu\left(\frac{k}{d}\right)\left[
\sum_{m_1+2m_2+\cdots+nm_n=d} (-1)^{s_n}  d(m!) 
\prod_{j=1}^{n} \frac{(b_{n,n-j})^{m_j}}{ (m_j) !}  \right],
\]
where $m_j$ are non-negative integers, $s_n=\sum_{i=1}^{[n/2]}m_{2i}$, 
$m=\sum_{i=1}^{n}m_{i}-1$, and 
$b_{n,j}$ are the Lah numbers for $G_n=\vP_n$ and the Stirling 
numbers of the second kind for $G_n=\vP_n^+$.
\end{corollary}

\begin{proof}
As mentioned in \S\ref{subsec:coho},  work of Bartholdi et al.~\cite{Bartholdi-E-E-R} 
and Lee \cite{Lee} shows that the cohomology algebra $A=H^*(G_n;\C)$ 
is Koszul.  Furthermore, it follows from Theorem \ref{thm:beer lee} that 
$A\cong U(\gr(G_n))^!$. 
Hence, by the Koszul duality formula mentioned in 
Proposition \ref{prop:kd}, we have that $\Hilb(U(\gr(G_n),-t))\cdot \Poin(G_n,t)=1$, 
thereby showing that $G_n$ satisfies the hypothesis of Lemma \ref{lem:PBW}.
Using now formulas \eqref{eq:hilbseriesvP} and \eqref{eq:lcs ranks} yields  
the desired expression for the LCS ranks of the groups $G_n$.
\end{proof}

\begin{remark}
\label{rmk:Universal}
The referee asked whether it is possible to 
compute explicitly the coefficients 
of the series $\Hilb(U(\gr(G_n),t))=1/\Poin(G_n,-t)$, 
where recall the polynomials $\Poin(G_n,t)$ are given 
in \eqref{eq:hilbseriesP} and \eqref{eq:hilbseriesvP}.  
As pointed out by the referee, the classical identity 
$\prod_{k=1}^{n}(1-kt)^{-1}=\sum_{i\geq 0}S(n+i,n)t^i$, 
where $S(n+i,n)$ are the Stirling 
numbers of the second kind (see \eqref{eq:stirling2} 
and \cite[(1.94c)]{Stanley}), implies that
\begin{equation}
\label{eq:hilbP}
\Hilb(U(\gr(P_{n+1})),t)=\sum_{i\geq 0}S(n+i,n)t^i.
\end{equation}
Direct computation shows that 
$\Hilb(U(\gr(\vP^+_{3})),t)=\sum_{i\ge 0} a_{2i+2} t^i$,  
where $a_n$ is the $n$-th Fibonacci number, 
while $\Hilb(U(\gr(\vP_{3})),t)=\sum_{i\ge 0} \sqrt{6}\,^i U_i(\sqrt{6}/2) t^i$,  
where $U_n(x)$ is the $n$-th Chebyshev 
polynomial of the second kind.  We do not know the corresponding 
Hilbert series expansions for the groups $\vP^+_n$ and $\vP_n$ with $n\ge 4$. 
\end{remark}

\subsection{Non-mild presentations}
\label{subsec:mild}
Again, let $G_n$ denote any one of the pure braid-like groups 
$P_n$, $\vP_n$, or $\vP_n^+$.
Recall that $G_n$ is graded-formal, and $\gr(G_n)\cong \fL(G_n)$. 
However, as we show next, the groups $G_n$ are not mildly presented, 
except for small $n$.

\begin{prop}
\label{prop:mild}
The pure braid groups $P_n$ and the pure virtual braid groups 
$\vP_n$ and $\vP_n^+$  admit  mild presentations if and only if $n\leq 3$. 
\end{prop}

\begin{proof}
Let $G_n$ denote any of the aforementioned groups.
Then $G_n$ is a commutator-relators group, and 
the universal enveloping algebra of the associated graded Lie 
algebra is Koszul.  From formulas \eqref{eq:hilbseriesP} and \eqref{eq:hilbseriesvP}, 
for $n\leq 3$, Anick's criterion \eqref{eq:anick criterion} is satisfied. 
Hence, $G_n$ has a mild presentation for $n\leq 3$.

Now suppose $n\ge 4$.  
Using formulas \eqref{eq:hilbseriesP} and \eqref{eq:hilbseriesvP} once 
again, we see that the third Betti numbers of these groups are given by 
\begin{equation}
\label{eq:b3}
b_3(P_n)=s(n,n-3),\quad  b_3(\vP_n)=L(n,n-3),\quad b_3(\vP_n^+)=S(n,n-3).
\end{equation}
Thus, $\dim H^3(G_n,\C)>0$ for $n\geq 4$.  
The claim now follows from Proposition \ref{lem:mild}
and the fact that $H^*(G_n,\C)=U(\gr(G_n,\C))^!=\Ext_U(\C,\C)$.
\end{proof}

\section{Malcev Lie algebras and filtered formality}
\label{sec: Formality} 
 
In this section we study the formality properties of the 
pure virtual braid groups, and prove Theorem \ref{thm:intro formal} 
from the Introduction. 
We start with a review of the  relevant formality notions, 
following Quillen \cite{Quillen69} and Sullivan \cite{Sullivan}.  
For more details and references, we refer to \cite{SW1}.
 
\subsection{Malcev Lie algebras and $1$-formality}
\label{subsec:malcev}
Let $G$ be a finitely generated group. 
The group-algebra $A=\C[{G}]$ is a Hopf algebra,
with comultiplication $\Delta\colon A\otimes A\to A$ given by 
$\Delta(g)=g\otimes g$ for $g\in G$, and counit the 
map $\varepsilon\colon A\to \C$ given by $\varepsilon(g)=1$. 
Let $\widehat{A}=\varprojlim_r A/I^r$ be the completion of $A$ 
with respect to the $I$-adic filtration, where $I=\ker(\varepsilon)$
is the augmentation ideal. 
The set $\fm(G)$ of all primitive elements  (that is, the set of  all $x\in \widehat{A}$ 
such that 
$\widehat{ \Delta} x=x\widehat{\otimes} 1+1\widehat{\otimes} x$), with Lie bracket 
$[x,y]=xy-yx$, and endowed with the induced filtration, 
is known as the \textit{Malcev Lie algebra}\/ of $G$.  

Let  $\widehat{\gr}(G)$ be the completion of the associated 
graded Lie algebra of $G$ with respect to the lower central 
series filtration. 
The group $G$ is said to be {\em filtered-formal}\/ if $\fm(G)$ 
is isomorphic (as a filtered Lie algebra) to $\widehat{\gr}(G)$, 
or,  equivalently, if there exists a morphism of filtered Lie algebras,  
$\phi\colon\fm(G)\to \widehat{\gr}(G)$, such that 
$\gr_1(\phi)$ is an isomorphism.
The group $G$ is said to be \emph{$1$-formal}\/ if $\fm(G)\cong \widehat{\fh}(G)$.  
Clearly, $G$ is $1$-formal if and only if it is both graded-formal 
and filtered-formal.  

It is easily seen that $G$ is $1$-formal if $b_1(G)$ is $0$ or $1$. 
Furthermore, if $F$ is a finitely generated free group, 
then $\fm(F)$ is the completed free Lie algebra of the same rank, 
and thus $F$ is $1$-formal.  Other well-known examples of $1$-formal 
groups include the fundamental groups of compact K\"ahler manifolds 
(\cite{DGMS}) and the fundamental groups of complements of complex 
algebraic hypersurfaces (\cite{Kohno}).  As noted by Bezrukavnikov \cite{Bez}, 
the group $P_{g,n}$ of pure braids on a surface of genus $g$ is always 
filtered-formal; it is also graded-formal for $g\geq 2$ or $g=1$ and $n\le 2$, 
but not for $g=1$ and $n\ge 3$ (see also \cite{DPS}). 

\subsection{Formality, group operations, and resonance}
\label{subsec:propagation}

As shown in \cite{DPS},  the $1$-formality property of groups is preserved 
under (finite) products and the coproducts. 
We sharpen these results in \cite{SW1}, as follows. 

\begin{theorem}[\cite{SW1}]
\label{thm:product formality}
Let  $G_1$ and $G_2$ be two finitely generated groups. 
The following conditions are equivalent.
\begin{enumerate}
\item  $G_1$ and $G_2$ are graded-formal 
(respectively, filtered-formal, or $1$-formal).
\item   $G_1* G_2$  is graded-formal 
(respectively, filtered-formal, or $1$-formal).
\item   $G_1\times G_2$ is graded-formal 
(respectively, filtered-formal, or $1$-formal).
\end{enumerate}
\end{theorem}

The next theorem shows how the formality properties are preserved 
under split inclusions.

\begin{theorem}[\cite{SW1}]
\label{thm:formality}
Let $N$ be a subgroup of a finitely generated group $G$. 
Suppose there is a split monomorphism 
$\iota\colon N\rightarrow G$. If $G$ is $1$-formal, then $N$ is also $1$-formal. 
A similar statement holds for graded-formality and filtered-formality.
\end{theorem}

An important obstruction to $1$-formality is provided by the higher-order 
Massey products, but we will not make use of it in this paper. 
Instead, we will use another, better suited obstruction to $1$-formality, 
which is provided by 
the following theorem. 

\begin{theorem}[\cite{DPS}]
\label{thm:tangentcone}
Let G be a finitely generated, $1$-formal group. Then all irreducible 
components of $\cR^1_d(G)$ are rationally defined linear subspaces 
of $H^1(G,\C)$, for all $d\ge 0$.
\end{theorem}

\subsection{Formality properties of $\PV_n$ and  $\PV^+_n$}
\label{subsec:f-pvn}
Recall that the pure virtual braid groups $\PV_n$ and $\PV^+_n$ 
are graded-formal for all $n\ge 1$. Furthermore, $\PV^+_2=\Z$ 
and $\PV_2=F_2$, and so both are $1$-formal groups.

\begin{lemma}
\label{thm:PV3formal}
The groups $\PV^+_3$ and $\PV_3$ are both $1$-formal. 
\end{lemma}

\begin{proof}
As shown in \cite{DPS} (see also \cite{SW1}), the free 
product of two $1$-formal groups is $1$-formal. 
Hence $\PV^+_3\cong \Z^2*\Z$ is also $1$-formal.

Since the pure braid group $P_4\cong \overline{P}_4\times \Z$ 
is $1$-formal, Theorem \ref{thm:formality} ensures that the subgroup 
$\overline{P}_4$ is also $1$-formal. On the other hand, we 
know from Lemma \ref{lem:isogroups} that 
$\PV_3\cong \overline{P}_4\ast \Z$.  Thus, by, 
Theorem \ref{thm:formality}, the group $\PV_3$ is $1$-formal.
\end{proof}

 \begin{lemma}
 \label{lem:pv4+}
The group $\PV^+_4$ is not $1$-formal.
\end{lemma}

\begin{proof}
As shown in Lemma \ref{lem:resonancePV4+}, the resonance variety 
$\cR^1_1(\PV_4^+)$ is an irreducible subvariety of $H^1(\PV_4^+,\C)$.  
Thus, this variety doesn't decompose into a finite union of linear subspaces,  
and so, by Theorem \ref{thm:tangentcone}, the group 
$\PV^+_4$ is not $1$-formal.
\end{proof}

\begin{theorem}
\label{thm:pvn nonformal}
The groups $\PV^+_n$ and $\PV_n$ are not $1$-formal for $n\geq4$.
\end{theorem}

 \begin{proof}
As shown in Lemma \ref{lem:split}, there is a split injection from $\PV^+_n$ 
to $\PV^+_{n+1}$. Since $\PV^+_4$ is not $1$-formal,
Theorem \ref{thm:formality}, then, insures that the groups $\PV^+_n$ 
are not $1$-formal for $n\geq4$.
 
By Corollary \ref{cor:split}, there is a split monomorphism $\PV^+_n\rightarrow \PV_n$.
From Theorem \ref{thm:pvn nonformal}, we know that the group $\PV^+_n$ is 
not $1$-formal for $n\geq 4$.
Therefore, by Theorem \ref{thm:formality}, the group $\PV_n$ is not 
 $1$-formal for $n\geq4$.
\end{proof}
 
\begin{corollary}
\label{cor:pvn}
The groups $\PV^+_n$ and $\PV_n$ are not filtered formal for $n\geq4$.
\end{corollary}

To summarize, the groups $\PV_n$ and $\PV_n^+$ are always 
graded-formal.  Furthermore, they are $1$-formal (equivalently, 
filtered-formal) if and only if $n\le 3$.  This completes the proof 
of Theorem \ref{thm:intro formal} from the Introduction. 

\section{Chen Lie algebras and Alexander invariants}
\label{sec:ChenAlex}

In this section, we discuss the relationship between the Chen Lie 
algebra and the Alexander invariant of a finitely generated group. 

\subsection{Chen Lie algebras and Chen ranks}
\label{subsec:ChenLie}

Let $G$ be a finitely generated group. 
The \textit{Chen Lie algebra}\/ of $G$, as defined by Chen \cite{Chen51}, 
is the associated graded Lie algebra of its second derived quotient, 
$G/G^{\prime\prime}$. The projection $\pi\colon G\surj G/G^{\prime\prime}$ induces 
an epimorphism, $\gr(\pi) \colon \gr(G)\surj \gr(G/G^{\prime\prime})$. It is readily 
verified that $\gr_k(\pi)$ is an isomorphism for $k\leq 3$.

The integers $\theta_k(G):=\rank(\gr_k(G/G^{\prime\prime}))$ are called 
the {\em Chen ranks}\/ of $G$.

\begin{lemma}
\label{lem:ChenLie}
The Chen Lie algebra of the product of two groups $G_1$ and $G_2$ is 
isomorphic to the direct sum 
$\gr(G_1/G_1^{\prime\prime})\oplus \gr(G_2/G_2^{\prime\prime})$, as graded Lie algebras.
\end{lemma}

\begin{proof}
The canonical projections $G_1\times G_2\to G_i$ for $i=1,2$ restrict 
to homomorphisms on the second derived subgroups, 
$(G_1\times G_2)^{\prime\prime}\to G_i^{\prime\prime}$.
Hence, there is an epimorphism $\phi\colon 
G_1\times G_2/(G_1\times G_2)^{\prime\prime}\to G_1/G_1^{\prime\prime}\times
G_2/G_2^{\prime\prime}$, inducing an epimorphism 
\begin{equation}
\label{eq:grphi}
\xymatrix{\gr(\phi)\colon 
\gr((G_1\times G_2)/(G_1\times G_2)^{\prime\prime})\ar@{->>}[r]& 
\gr(G_1/G_1^{\prime\prime})\oplus \gr(G_2/G_2^{\prime\prime})}.
\end{equation}
By \cite[Corollary 1.10]{Cohen-Suciu99T}, we have that 
\begin{equation}
\label{eq:thetapro}
\theta_k(G_1\times G_2)=\theta_k(G_1)+\theta_k(G_2). 
\end{equation}
Hence, the homomorphism $\gr(\phi)$
is an isomorphism of graded Lie algebras.
\end{proof}

In \cite[Theorem 8.4]{SW1}, we prove the following.

\begin{theorem}[\cite{SW1}]
\label{thm:sw1}
Let $G$ be a finitely generated group.   The canonical projection 
$G\surj G/G^{\prime\prime}$ induces then an epimorphism of graded 
Lie algebras,
\begin{equation}
\label{eq:surjChen}
\xymatrixcolsep{20pt}
\xymatrix{\Phi\colon ~~\gr(G)/\gr(G)^{\prime\prime} \ar@{->>}[r]& 
\gr(G/G^{\prime\prime})}.
\end{equation}
Furthermore, if $G$ is filtered-formal, then the above map is an isomorphism.
\end{theorem}

In the case when $G$ is $1$-formal, this theorem recovers a result from 
\cite{Papadima-Suciu04}, which insures there is a graded Lie algebra isomorphism 
$\fh(G)/\fh(G)^{\prime\prime} \cong  \gr(G/G^{\prime\prime})$.
 
\subsection{Alexander invariants}
\label{subsec:alex inv}
Once again, let $G$ be a finitely generated group.  Let us consider the 
$\C$-vector space $H_1(G',\C)=G^{\prime}/G^{\prime\prime} \otimes \C$. 
This vector space can be viewed as a (finitely generated) module over the group algebra 
$\C[H]$, with the abelianization $H=G/G^{\prime}$ acting on $G'/G''$ by conjugation.  
Following \cite{Massey80}, we denote this module by $B_{\C}(G)$, or $B(G)$ for short, and 
call it the {\em Alexander invariant}\/ of $G$.  
We refer to \cite{Massey80, Cohen-Suciu99T, Papadima-Suciu04} for 
ways to compute presentations for the module $B(G)$ in 
various degrees of generality.

The module $B=B(G)$ may be 
filtered by powers of the augmentation ideal, $I=\ker(\varepsilon\colon \C[H] \to \C)$, 
where  $\varepsilon$ is the ring map defined by $\varepsilon(h)=1$ for all $h\in H$. 
The associated graded module,
\begin{equation}
\label{eq:grb}
\gr(B)=\bigoplus_{k\ge 0} I^k   B/I^{k+1}  B,
\end{equation}
then, is a module over the graded ring $\gr(\C[H])=\bigoplus_{k\ge 0} I^k /I^{k+1}$.  
We call this module the {\em associated graded Alexander invariant}\/ of $G$. 

Work of W.~Massey \cite{Massey80} implies that 
the map $j\colon G'/G''\to G/G''$ restricts to isomorphisms 
\begin{equation}
\label{eq:Masseyiso}
\xymatrix{  I^k B  \ar[r] &\Gamma_{k+2}(G/G'')}
\end{equation}
for all $k\geq 0$. 
Taking successive quotients of the respective filtrations 
and tensoring with $\C$, we obtain isomorphisms
\begin{equation}
\label{eq:MasseyAlex}
\xymatrix{\gr_k(j)\colon \gr_k(B(G)) \ar[r] &\gr_{k+2}(G/G'')} \textrm{ for } k\geq 0.
\end{equation}
Consequently, the Chen ranks of $G$  can be expressed  in terms 
of the Hilbert series 
of the graded module $\gr(B(G))$, as follows:
\begin{equation}
\label{eq:MasseyChen}
\sum\limits_{k\geq 2}\theta_{k}(G)\cdot t^{k-2}=\Hilb (\gr(B(G)),t).
\end{equation}

\begin{remark}
\label{remk:comp}
If the group $G=\langle x_1,\dots,x_n \mid r_1,\dots, r_m\rangle$ 
is a finitely presented, commu\-tator-relators group, then the Hilbert 
series of the module $\gr(B(G))$ may be computed using the
algorithm from \cite{Cohen-Suciu95, Cohen-Suciu99}.  To start 
with, identify $\C[H]$ with $\Lambda=\C[t_1^{\pm 1}, \dots, t_n^{\pm 1}]$.  
The Alexander invariant of $G$ admits then a finite 
presentation for the form 
\begin{equation}
\label{AI presentation}
\xymatrixcolsep{26pt}
\xymatrix{
\Lambda^{\binom{n}{3}}\oplus \Lambda^{m} \ar^(.6){\delta_3+\nu_G}[r]& 
\Lambda^{\binom{n}{2}}  \ar[r]& B(G) \ar[r]& 0
}, 
\end{equation}
Here, $\delta_i$ is the $i$-th differential in the standard Koszul 
resolution of $\C$ over $\Lambda$,
and $\nu_G$ is a map satisfying $\delta_2\circ \nu_G=D_G$,
where $D_G$ is the abelianization of the Jacobian matrix 
of  Fox derivatives of the relators.  Next, one computes a 
Gr\"obner basis for the module $B(G)$, in a suitable monomial ordering.  
An application of the standard tangent cone algorithm yields then  
a presentation for $\gr(B(G))$, from which ones computes  
the Hilbert series of $\gr(B(G))$.  Finally, the Chen ranks
of $G$ are given by formula \eqref{eq:MasseyChen}.
\end{remark}

\subsection{Chen ranks of the free groups}
\label{subsec:chenF}

As shown in \cite{Chen51}, the Chen ranks of the free group $F_n$ are given by 
$\theta_1(F_n)=n$ and 
\begin{equation}
\label{eq:ChenFreeGroups}
\theta_k(F_n)=\binom{n+k-2}{k} (k-1)  ~\textrm{ for }~  k\ge 2 .
\end{equation}
Equivalently, 
by Massey's formula \eqref{eq:MasseyChen}, the Hilbert series 
for the associated graded Alexander invariant of $F_n$ is given by 
\begin{equation}
\label{eq:bfn}
\Hilb (\gr(B(F_n)),t)=\frac{1}{t^2} \cdot \left(1 - \frac{1-nt}{(1-t)^n} \right),
\end{equation}
an identity which can also be verified directly, by using the fact that 
$B(F_n)$ is the cokernel of the third boundary map of the Koszul resolution 
$\bigwedge (\Z^n) \otimes \C[\Z^n]$.

From formula \eqref{eq:bfn}, we see that the generating and exponential 
generating functions for the Hilbert series of the associated graded Alexander 
invariants of the sequence of free 
groups $\mathbf{F}=\{ F_n\}_{n\geq 1}$ are given by
\begin{align}
\label{eq:chen free}
\sum_{n=1}^{\infty} \Hilb (\gr(B(F_n)),t)\cdot u^n &=\frac{u^2}{(1-u)(1-t-u)^2}, 
\\  \notag
\sum_{n=1}^{\infty} \Hilb (\gr(B(F_n)),t)\cdot \frac{u^n}{n!} &=
\frac{e^u}{t^2}+\frac{e^{u/(1-t)}}{t^2} \left(\frac{tu}{1-t}-1\right).
\end{align}
 
\subsection{Chen ranks of the pure braid groups}
\label{subsec:chenPure}
A comprehensive algorithm for computing the Chen ranks of finitely 
presented groups was developed in \cite{Cohen-Suciu95, Cohen-Suciu99}, 
leading to the following expressions for the Chen ranks 
of the pure braid groups:
\begin{equation}
\label{eq:chenpn}
\text{$\theta_1(P_n)=\binom{n}{2}$,\:  $\theta_2(P_n)=\binom{n}{3}$,\:
and\: $\theta_k(P_n)=(k-1)\binom{n+1}{4}$\: for $k\ge 3$},   
\end{equation}
or, equivalently, 
\begin{equation}
\label{eq:bpn}
\Hilb (\gr(B(P_n)),t)=\binom{n+1}{4} \frac{1}{(1-t)^2} -\binom{n}{4}.
\end{equation}

It follows that the two generating functions for the Hilbert series of 
the associated graded Alexander invariants of the sequence of pure braid 
groups $\mathbf{P}=\{ P_n\}_{n\geq 1}$ are given by 
\begin{align}
\label{eq:chenpnseries}
\sum_{n=1}^{\infty} \Hilb (\gr(B(P_n)),t)\cdot u^n &=
\frac{u^3}{(1-u)^5}\left(\frac{1}{(1-t)^2}-u\right), 
\\ \notag
\sum_{n=1}^{\infty} \Hilb (\gr(B(P_n)),t)\cdot \frac{u^n}{n!} &=
\frac{e^uu^3}{24} \left( \frac{u+4}{(1-t)^2} -u
\right).
\end{align}

\section{Infinitesimal Alexander invariants}
\label{sec:infalex}

We now turn to the infinitesimal Alexander invariants,  
$\fB(\fh(G))$ and $\fB(\gr(G))$, and the way they relate to the associated 
graded Alexander invariant, $\gr(B(G))$.

\subsection{The infinitesimal Alexander invariant of a Lie algebra}
\label{subsec:inf alex}
We start in a more general context. 
Let  $\fg$ be a finitely generated graded Lie algebra, with 
graded pieces $\fg_k$, for $k\ge 1$.  
Then both the derived algebra, $\fg'$, and the second derived 
algebra $\fg''=(\fg')'$, are graded sub-Lie algebras.  Thus, 
the maximal metabelian quotient, $\g/\g''$, is in a natural 
way a graded Lie algebra, with derived subalgebra $\g'/\g''$.  
Define the Chen ranks of $\fg$ to be 
\begin{equation}
\label{eq:inf chen ranks}
\theta_k(\fg)=\dim (\fg/\fg'')_k.
\end{equation}

Following \cite{Papadima-Suciu04}, we associate to $\g$ a graded 
module over the symmetric algebra $S=\Sym(\fg_1)$, as follows.  
The adjoint representation of $\g_1$ on $\g/\g''$ 
defines an $S$-action on $\g'/\g''$, given by $h\cdot \bar{x}=\overline{[h,x]}$, 
for $h\in \g_1$ and $x\in \g'$.  Clearly, this action is compatible with the 
grading on $\g'/\g''$.  We then let the {\em infinitesimal Alexander invariant}\/ 
of $\g$ to be the graded $S$-module 
\begin{equation}
\label{eq:inf alex}
\fB(\fg)=\fg^{\prime}/\fg^{\prime\prime}.
\end{equation} 

Assume now that the graded Lie algebra $\fg=\bigoplus_{k\ge 1} \fg_k$ is generated 
in degree $1$.  We then have $\fg'=\bigoplus_{k\ge 2} \fg_k$.  Thus, since the grading 
for $S$ starts with $S_0=\C$, we are led to define the grading on $\fB(\g)$ as 
\begin{equation}
\label{eq:fb grading}
 \fB(\g)_{k} = (\fg'/\fg'')_{k+2},  \textrm{ for } k\geq 0.
 \end{equation}

We then have the following `infinitesimal' version of Massey's 
formula \eqref{eq:MasseyChen}.

\begin{prop}
\label{prop:MC Lie alg}
Let $\fg$ be a finitely generated, graded Lie algebra $\fg$ 
generated in degree $1$.  Then the Chen ranks of $\fg$ 
are given by
\begin{equation}
\label{eq:MC Lie alg}
\sum\limits_{k\geq 2}\theta_{k}(\fg)\cdot t^{k-2}=\Hilb (\fB(\fg),t).  
\end{equation}
\end{prop}

\begin{proof}
Since $\fg$ is generated in degree $1$, we have that 
$ \fg/\fg' \cong \fg_1$.  Using now 
the exact sequence of graded Lie algebras $0\to \fg'/\fg'' \to \fg/\fg''  \to \fg/\fg'  \to 0$, 
we find that $(\fg/\fg'')_k = (\fg'/\fg'')_k$ for all $k\ge 2$.  The claim then 
follows from \eqref{eq:inf chen ranks} and \eqref{eq:fb grading}.
\end{proof}

\subsection{The infinitesimal Alexander invariants of a group}
\label{subsec:bhbgr}

Let again $G$ be a finitely generated group. Denote by $H=G_{\ab}$ its 
abelianization, and identify $\h_1(G)=\gr_1(G)$ with $H\otimes \C$.  
Finally, set $S=\Sym(H\otimes \C)$.  The procedure outlined in 
\S\ref{subsec:inf alex} yields two $S$-modules attached to $G$.  

The first one is $\fB(G)=\fB(\fh(G))$, the infinitesimal 
Alexander invariant of the holonomy Lie algebra of $G$. 
(When $G$ is a finitely presented, commutator-relators 
group, this $S$-module coincides with the `linearized Alexander invariant' 
from \cite{Cohen-Suciu99T, Matei-Suciu00},  see 
\cite{Papadima-Suciu04}).   
The second one is $\fB(\gr(G))$, the infinitesimal Alexander invariant of the 
associated graded Lie algebra of $G$.  The next result provides 
a natural comparison map between these $S$-modules.

\begin{prop}
\label{prop:alexinv compare}
The  canonical epimorphism $\Psi\colon \fh(G)\surj \gr(G)$ from 
\eqref{eq:holonomysurj} induces an epimorphism of $S$-modules, 
\[
\xymatrix{ \psi\colon \fB(\fh(G)) \ar@{->>}[r] &\fB(\gr(G))}.
\]
Moreover, if $G$ is graded-formal, then $\psi$ is an isomorphism. 
\end{prop}

\begin{proof}
The graded Lie algebra map  $\Psi\colon \fh(G)\surj \gr(G)$ preserves 
derived series, and thus induces an epimorphism 
$\psi\colon \fh(G)'/\fh(G)'' \surj  \gr(G)'/\gr(G)''$. By the discussion 
from \S\ref{subsec:inf alex}, this map can also be viewed 
as a map $\psi\colon \fB(\fh(G)) \surj \fB(\gr(G))$ 
of graded $S$-modules.  

Finally, if $G$ is graded-formal, i.e., if $\Psi$ is an isomorphism, then 
clearly $\psi$ is also an isomorphism.
\end{proof}

\begin{remark}
\label{rem:comp-bis}
For a finitely presented group $G$, a finite presentation for the $S$-module 
$\fB(\fh(G))$ is given in \cite{Papadima-Suciu04}.  This presentation may be used 
to compute the holonomy Chen ranks $\theta_k(\fh(G))$ 
from the Hilbert series of $\fB(\fh(G))$, using an approach 
analogous to the one described in Remark \ref{remk:comp}.  
We refer to \cite{SW3} for more information on this subject, 
illustrated with detailed computations for the (upper) pure 
welded braid groups.
\end{remark}

\subsection{Another filtration on $G'/G''$}
\label{subsec:ind filt}
Next, we compare the module $\fB(\gr(G))$ to another, naturally defined $S$-module 
associated to the group $G$.  Let $\gr^{\widetilde{\Gamma}}(G'/G'')$ be the associated 
graded Lie algebra of $G'/G''$ with respect to the induced filtration 
\begin{equation}
\label{eq:inducedfil}
\widetilde{\Gamma}_k(G'/G''):=  (G'/G'') \cap \Gamma_k(G/G''). 
\end{equation}
The terms of this filtration fit into short exact sequences 
\begin{equation}
\label{eq:gamma exact}
\xymatrix{
0\ar[r]&\widetilde{\Gamma}_k(G'/G'')\ar[r]& \Gamma_k(G/G'')\ar[r]& \Gamma_k(G/G')\ar[r]& 0
}.
\end{equation}
Noting that $\Gamma_k(G/G')=0$ for $k\geq 2$, we deduce that 
\begin{equation}
\label{eq: TildeFiltration}
\widetilde{\Gamma}_k(G'/G'')= \Gamma_k(G/G''), \textrm{ for } k\geq 2.
\end{equation}

As before, it is readily checked that the adjoint representation of 
$\gr_1(G/G'')= H\otimes \C$ on $\gr(G/G'')$ induces an $S$-action 
on $\gr^{\widetilde{\Gamma}}(G'/G'')$, 
preserving the grading.  Hence, the  Lie algebra 
$\fC(G):=\gr^{\widetilde{\Gamma}}(G'/G'')$ can 
also be viewed as a graded module over $S$, by setting 
\begin{equation}
\label{eq:mc grading}
\fC(G)_{k} = \gr^{\widetilde{\Gamma}}_{k+2}(G'/G'').
\end{equation}

\begin{prop}
\label{prop:gamma compare}
The canonical morphism of graded Lie algebras  
$\Phi\colon \gr(G)/\gr(G)'' \surj \gr(G/G'')$  from \eqref{eq:surjChen}
induces an epimorphism of $S$-modules, 
\begin{equation*}
\xymatrix{ \varphi\colon \fB(\gr(G)) \ar@{->>}[r] &\fC(G)}.  
\end{equation*}
Moreover, if $G$ is filtered-formal, then $\varphi$ is an isomorphism. 
\end{prop}

\begin{proof}
The map $\Phi$ fits into the following commutative diagram 
of graded Lie algebras, 
\begin{equation}
\label{eq:ladder}
\xymatrix{
0\ar[r]& \gr(G)'/\gr(G)''\ar@{.>>}[d]^{\varphi} \ar[r] &\gr(G)/\gr(G)'' 
\ar@{->>}[d]^{\Phi}\ar[r] & \gr(G)/\gr(G)' \ar[d]^{\id}\ar[r] &0\phantom{,}\\
0\ar[r]& \gr^{\widetilde{\Gamma}}(G'/G'')\ar[r]^{\gr(j)} &\gr(G/G'') \ar[r] 
& \gr(G/G') \ar[r] &0.
}
\end{equation}

Thus, $\Phi$ induces a morphism of graded Lie algebras, 
$\varphi\colon \gr(G)'/\gr(G)''\to \gr^{\widetilde{\Gamma}}(G'/G'')$, 
as indicated above. By the Five Lemma, $\varphi$ is surjective.  
Observe that $\gr(G)/\gr(G)'\cong \gr(G/G')\cong H\otimes \C$ acts on 
both the source and target of $\varphi$ by adjoint representations. 
Hence, upon regrading according to \eqref{eq:fb grading} 
and \eqref{eq:mc grading}, the map $\varphi\colon \fB(\gr(G)) \surj \fC(G)$ 
becomes a morphism of $S$-modules.  

If $G$ is filtered-formal, then, according to Theorem \ref{thm:sw1}, 
the map $\Phi$ is an isomorphism of graded Lie algebras.
Hence, the induced map $\varphi$ is an isomorphism of $S$-modules.
\end{proof}

\subsection{Comparison with the associated graded Alexander invariant}
\label{subsec:assoc gr}

Finally, we identify the associated graded Alexander invariant $\gr(B(G))$ 
with the $S$-module $\fC(G)$ defined above. To do that, we first identify 
the respective ground rings.  

Choose a basis $\{x_1,\dots, x_n\}$ for the torsion-free part of $H=G_{\ab}$.  
We may then identify the group algebra $\gr(\C[H])$ with the polynomial algebra 
$R=\C[s_1,\dots,s_n]$, where $s_i$ corresponds to $\overline{x_i-1}\in I/I^2$, 
see Quillen \cite{Quillen68}.  On the other 
hand, we may also identify the symmetric algebra $S=\Sym(H\otimes \C)$ with the 
polynomial algebra $\C[x_1,\dots, x_n]$.   The desired ring isomorphism, 
$R\cong S$, is gotten by sending $s_i$ to $x_i$.

\begin{prop}
\label{thm:gradedinf} 
Under the above identification $R\cong S$, the graded $R$-module $\gr(B(G))$ 
is canonically isomorphic to the graded $S$-module $\fC(G)$.
\end{prop}

\begin{proof}
Recall from \S\ref{subsec:alex inv} that the inclusion map 
$j\colon G'/G''\to G/G''$ restricts to an 
isomorphism $I^k B(G)\to \Gamma_{k+2}(G/G'')$ for each $k\geq 0$. 
Using the induced filtration $\widetilde{\Gamma}$ from 
\eqref{eq:inducedfil} and the identification \eqref{eq: TildeFiltration},
we obtain $\C$-linear isomorphisms 
$ I^k B(G) \cong \widetilde{\Gamma}_{k+2}(G'/G'')$, for all $k\ge 0$. 
Taking the successive quotients of the respective filtrations 
and regrading according to  \eqref{eq:mc grading}, 
we obtain a $\C$-linear  isomorphism $\gr(B(G))\cong \fC(G)$.

Under the identification $\gr(\C[H])\cong R$, the associated graded 
Alexander invariant of $G$ may be viewed as a graded $R$-module, 
with $R$-action defined by 
\begin{equation}
\label{eq:R action}
s_i(\overline{z})=(\overline{x_i-1}) \overline{z}=\overline{x_izx_i^{-1}-z}=
\overline{[x_i,z]z-z}=\overline{[x_i,z]}+\overline{z}-\overline{z}=\overline{[x_i,z]},
\end{equation}
for all $z\in G'$.
(In this computation, we follow the convention from \cite{Massey80}, 
and view the Alexander invariant 
$B(G)=G'/G''$ as an additive group; however, when we consider
the induced filtration  $\widetilde{\Gamma}$ on 
$G'/G''$, we view it as a multiplicative subgroup of $G/G''$.)

Finally, recall that $\fC_{\bullet}(G)=\gr^{\widetilde{\Gamma}}_{\bullet -2}(G'/G'')$ 
is an $S$-module, with 
 $S$-action given by $x_i(\overline{z})=\overline{[x_i,z]}$.
Hence, the aforementioned isomorphism $R\cong  S$ identifies the 
$R$-module $\gr(B(G))$ with the $S$-module $\fC(G)$.
\end{proof}

\subsection{Discussion}
\label{subsec:discuss}
In the $1$-formal case, we obtain the following corollary, which can also 
be deduced from \cite[Theorem 5.6]{DPS}.

\begin{corollary}
\label{cor:dps iso}
Let $G$ be a $1$-formal group.  
Then $\gr(B(G))\cong \fB(G)$, as modules 
over the polynomial ring $S=\gr(\C[H])$.  
\end{corollary}

\begin{proof}
Follows at once from Propositions \ref{prop:alexinv compare}, 
\ref{prop:gamma compare}, and \ref{thm:gradedinf}.
\end{proof}

Using those propositions 
once again, we obtain another corollary. 

\begin{corollary}
\label{cor:thetatilde}
Let $G$ be a finitely generated group.  The following then hold. 
\begin{enumerate}
\item \label{th1}
$\theta_k(\gr(G))\leq \theta_k(\fh(G))$, with equality if $k\le  2$, or if
$G$ is graded-formal.
\item  \label{th2}
$\theta_k(G)\leq \theta_k(\gr(G))$, 
with equality if $k\le  3$, or if $G$ is filtered-formal.
\end{enumerate}
\end{corollary}

The graded-formality assumption from part \eqref{th1} of the above corollary 
is clearly necessary for the equality $\theta_k(\gr(G))= \theta_k(\fh(G))$ 
to hold for all $k$.  On the other hand, it is not clear whether the 
filtered-formality hypothesis from part \eqref{th2} is necessary 
for the equality $\theta_k(G)= \theta_k(\gr(G))$ to hold in general.
In view of several computations (some of which are summarized 
in the next section), we are led to formulate the following question. 

\begin{question}
\label{que:eq}
Suppose $G$ is a graded-formal group.  
Does the equality $\theta_k(G)= \theta_k(\gr(G))$  hold for all $k$?
\end{question}

\section{Chen ranks and resonance}
\label{sec:Chen}

In this section, we detect the relationship between the Chen ranks and the resonance
varieties. We also compute the Chen ranks of some (upper) pure virtual braid groups.

\subsection{The Chen ranks formula}
\label{subsec:Chenranks}
Let $G$ be a finitely presented, commutator-relators group.  
As shown in \cite{Matei-Suciu00}, for each $d\ge 1$, the 
resonance variety $\cR_d^1(G)$  coincides, at least away 
from the origin $0\in H^1(G;\C)$, 
with the support variety of the annihilator of $d$-th exterior power of 
the infinitesimal Alexander invariant;  that is,
\begin{equation}
\label{eq:supp bg}
\cR_d^1(G)= V\bigg(\Ann\Big(\bigwedge^d\fB(\fh(G))\Big)\bigg).
\end{equation}

The first author conjectured in \cite{Suciu01} that  
for $k \gg 0$, the Chen ranks of an arrangement group $G$ are given by
the \emph{Chen ranks formula}
\begin{equation}
\label{eq:Chenranksformula}
\theta_k(G)= \sum_{m\geq 2}h_m(G)\cdot \theta_k(F_m) , 
\end{equation}
where $h_m(G)$ is the number of $m$-dimensional irreducible components 
of $\cR_1^1(G)$.  A positive answer to this conjecture was given in 
\cite{Cohen-Schenck15} for a class of
$1$-formal groups which includes arrangement groups. 
To state this result, recall that a subspace $U\subset H^1(G;\C)$ is called 
{\em isotropic}\/ if the cup product $U\wedge U\to H^2(G;\C)$ is the zero map.

\begin{theorem}[\cite{Cohen-Schenck15}]
\label{thm:Chenrankformula}
Let $G$ be a finitely presented, commutator-relators $1$-formal group.
Assume that the components of $\cR_1^1(G)$ are isotropic, 
projectively disjoint, and reduced as a scheme. Then the 
Chen ranks formula holds for $G$.
\end{theorem}

Using Hilbert series of the Alexander invariants, the Chen ranks 
formula \eqref{eq:Chenranksformula} translates into the equivalent 
statement that
\begin{equation}
\label{eq:ChenHilb}
\Hilb(\gr(B(G)),t)-\sum_{m\geq 2}h_m(G)\cdot \Hilb(\gr(B(F_m)),t)
\end{equation}
is a polynomial (in the variable $t$). 

\begin{example}
\label{eq:chenpn res}
The pure braid group $P_n$ is an arrangement group, and thus satisfies 
the hypothesis of Theorem \ref{thm:Chenrankformula}.  In fact, 
we know from Proposition \ref{prop:res pn} that the resonance variety 
$\cR^1_1(P_n)$ has $\binom{n}{3}+\binom{n}{4}=\binom{n+1}{4}$ irreducible 
components, all of dimension $2$.  Thus, the computation from \eqref{eq:chenpn}
agrees with the one predicted by formula \eqref{eq:Chenranksformula}, 
for all $k\ge 3$.
\end{example}

As noted in \cite{Cohen-Schenck15}, it is easy to find examples of 
non-$1$-formal groups for which the Chen ranks formula does not 
hold.  For instance, if $G=F_2/\Gamma_3(F_2)$ is the Heisenberg group, then 
$\cR^1_1(G)=H^1(G,\C)=\C^2$, and thus formula \eqref{eq:Chenranksformula} 
would predict in this case that $\theta_k(G)=\theta_k(F_2)$ for $k$ large enough, 
where in reality $\theta_k(G)=0$ for $k\ge 3$.  On the other hand, here is an 
example of a finitely presented, commutator-relators group which satisfies 
the Chen ranks formula, yet which is {\em not}\/ $1$-formal. 

\begin{example}
\label{eq:31425}
Using the notation from \cite{Matei-Suciu00}, let $\cA = \cA(31425)$ 
be the `horizontal' arrangement of $2$-planes in $\R^4$ determined 
by the specified permutation, and let $G$ be the fundamental group 
of its  complement.  From \cite[Example 6.5]{Matei-Suciu00}, we 
know that $\cR_1^1(G)$ is an irreducible cubic hypersurface in 
$H^1(G,\C)=\C^5$.   Hence,  by Theorem \ref{thm:tangentcone}, 
the group $G$ is not $1$-formal (for an alternative argument,  
see \cite[Example 8.2]{DPS}). On the other hand, the singularity 
link determined by $\cA$ has all linking numbers $\pm 1$, and 
thus satisfies the Murasugi Conjecture, that is, $\theta_k(G)=\theta_k(F_4)$, 
for all $k\ge 2$, see \cite{Massey-Traldi, Papadima-Suciu04}.  
Therefore, the Chen ranks formula holds for the group $G$.
\end{example}

\subsection{Products and coproducts}
\label{subsec:prcopr}

We now analyze the way the Chen ranks formula \eqref{eq:Chenranksformula} 
behaves under (finite) products and coproducts of groups.

\begin{lemma}
\label{lem:productcom}
Let $G_1$ and $G_2$ be two finitely generated groups. 
The number of $m$-dimen\-sional irreducible components of the corresponding 
first resonance varieties satisfies the following additivity formula,
\begin{equation}
\label{eq:cm g1g2}
h_m(G_1\times G_2)=h_{m}(G_1)+h_{m}(G_2).
\end{equation}

\end{lemma}
\begin{proof}
We start by identifying the affine space $H^1(G_1\times G_2;\C)$ with 
$H^1(G_1;\C)\times H^1(G_2;\C)$.
Next, by Proposition \ref{prop:resProd}, we have that 
\begin{equation}
\label{eq:r1g1g2}
\cR^1_1(G_1\times G_2) = \cR^1_1(G_1)\times \{0\} 
\cup \{0\}\times \cR^1_1(G_2).
\end{equation}

Suppose $\cR^1_1(G_1)=\bigcup_{i=1}^s A_i$ and 
$\cR^1_1(G_2)= \bigcup_{j=1}^t B_j$ are the decompositions 
into irreducible components for the respective varieties.
Then $A_i\times \{0\}$ and $\{0\}\times B_j$ are irreducible 
subvarieties of $\cR^1_1(G_1\times G_2)$.  Observe now that 
$\cR^1_1(G_1)\times \{0\}$ and  $\{0\}\times \cR^1_1(G_2)$
intersect only at $0$. It follows that 
\begin{equation}
\label{eq:comps r1g1g2}
\cR^1_1(G_1\times G_2) = \bigcup_{i=1}^s A_i\times \{0\}
 \cup  \bigcup_{j=1}^t \{0\}\times B_j
\end{equation}
is the irreducible decomposition for the first resonance variety 
of $G_1\times G_2$.  The claimed additivity formula follows.
\end{proof}

\begin{corollary}
\label{cor:productChenformula}
If both $G_1$ and $G_2$ satisfy the Chen ranks formula,
then $G_1\times G_2$ also satisfies the Chen ranks formula.
\end{corollary}

\begin{proof}
Follows at once from formulas \eqref{eq:thetapro} and \eqref{eq:cm g1g2}.  
\end{proof}

However, even if both $G_1$ and $G_2$ satisfy the Chen ranks formula,
the free product $G_1* G_2$ may not satisfy this formula. We illustrate 
this phenomenon with an infinite family of examples.

\begin{example}
\label{ex:chen ranks coprod}
Let $G_n=\Z* \Z^{n-1}$. Clearly, both factors of this free product satisfy the 
Chen ranks formula;  in fact, both factors satisfy the hypothesis of Theorem 
\ref{thm:Chenrankformula}.  
Moreover, $G_n$ is $1$-formal  and $\cR^1(G_n)$ is projectively disjoint and 
reduced as a scheme.  Using Theorem 4.1(3) and Lemma 6.2 from 
\cite{PS-mathann}, a short computation reveals that   
\begin{equation}
\label{eq:zztop}
\sum_{k\ge 2} \theta_k(G_n) t^k = t\, \frac{1-(1-t)^{n-1} }{(1-t)^n}.
\end{equation}

On the other hand, if $n\ge 2$,  then $\cR^1_1(G_n)= H^1(G_n, \C)$, 
by Proposition \ref{prop:res}. Thus, formula \eqref{eq:Chenranksformula} 
would say that $\theta_k(G_n)=\theta_k(F_n)$ for $k\gg 0$.  However, 
comparing formulas \eqref{eq:bfn} and \eqref{eq:zztop}, we find that 
\begin{equation}
\label{eq:fangs}
\theta_k(F_n)-\theta_k(G_n)=\sum_{i=2}^k\theta_i(F_{n-1}).
\end{equation}

Hence, if $n\ge 3$, the group $G_n$ does {\em not}\/ 
satisfy the Chen ranks formula.  Note that $G_n$ also does not satisfy 
the isotropicity hypothesis of Theorem \ref{thm:Chenrankformula}, 
since the restriction of the cup product to the factor $\Z^{n-1}$ is 
nonzero, again provided that $n\ge 3$. 
\end{example}

\subsection{Chen ranks of $vP_3$ and $\vP_3^+$}
\label{subsec:chen virtual}

We now return to the pure virtual braid groups, and study their Chen ranks. 
Recall that $\vP_2^+=\Z$ and $\vP_2=F_2$, so we may as well assume $n\ge 3$. 
We start with the case $n=3$.  

\begin{prop}
\label{prop:chenvp3+}
The groups $\vP_3^+$ and $\vP_3$ do not satisfy the Chen ranks formula, 
despite the fact that they are both $1$-formal, and their first resonance varieties 
are projectively disjoint and reduced as schemes.
\end{prop}

\begin{proof}
Recall that $\vP_3^+\cong \Z^2* \Z$. Thus, the claim for $\vP_3^+$ 
is handled by the argument from Example \ref{ex:chen ranks coprod}.  

Next,  recall that $\vP_3 \cong \overline{P}_4*\Z$.  
We know from Lemma \ref{thm:PV3formal} that $\vP_3$ 
is $1$-formal. Furthermore, we know from Proposition \ref{prop:vP3}
that $\cR_1^1(\vP_3)=H^1(\vP_3,\C)$.  Clearly, this variety is 
projectively disjoint and reduced as a scheme. 
Using the algorithm described in Remark \ref{remk:comp}, we find that 
the Hilbert series of the associated graded Alexander invariants of 
$\vP_3$ is given by
\begin{align}
\label{eq:chenvp3}
\Hilb (\gr(B(\vP_3)),t)&= (9-20t+15t^2-4t^4+t^5)/(1-t)^6. 
\end{align}

On the other hand, as noted above,  $\cR^1_1(\vP_3)=\C^6$.
Using \eqref{eq:bfn} and \eqref{eq:chenvp3}, we compute
\begin{align}
\label{eq:hilbgrbp3}
\Hilb (\gr(B(\vP_3)),t)-\Hilb (\gr(B(F_6)),t)&=(-6 + 6 t^3 - 5 t^4 + t^5)/(1 - t)^6. 
\end{align}
Since this expression is {\em not}\/ a polynomial in $t$, we conclude that 
formula \eqref{eq:Chenranksformula} does not hold for $\vP_3$, 
and this ends the proof.
\end{proof}

\begin{remark}
\label{rem:isotropicity}
The resonance varieties of $\vP_3^+$ and $\vP_3$ are not isotropic, 
since both groups have non-vanishing cup products stemming from the 
subgroups $\Z^2$ and $\overline{P}_4$, respectively. Thus, once again, 
the groups $\vP_3^+$ and $\vP_3$ illustrate the necessity of the isotropicity 
hypothesis from Theorem \ref{thm:Chenrankformula}.
\end{remark}

\subsection{Holonomy Chen ranks of $vP_n$ and $\vP_n^+$}
\label{subsec:holo chen virtual}
We conclude with a summary of what else we know about the Chen ranks 
of the pure virtual braid groups, as well as the Chen ranks of the 
respective holonomy Lie algebras and associated graded Lie algebras. 

\begin{prop}
\label{prop:ChenranksvP}
The following equalities of Chen ranks hold.
\begin{enumerate}
\item \label{q1}
$\theta_k(\fh(\vP_n^+))=\theta_k(\gr(vP_n^+))$ and 
$\theta_k(\fh(\vP_n))=\theta_k(\gr(vP_n))$, for all $n$ and $k$.
\item \label{q2}
$\theta_k(\fh(\vP_n^+))=\theta_k(vP_n^+)$  for $n\leq 6$ and all $k$.
\item \label{q3}
$\theta_k(\fh(\vP_n))=\theta_k(vP_n)$  for $n\leq 3$ and all $k$. 
\end{enumerate}
\end{prop}

\begin{proof}
\eqref{q1}
Recall from Theorem \ref{thm:beer lee} that the pure virtual braid 
groups $vP_n$ and $vP_n^+$ are graded-formal.  Therefore, 
by Corollary \ref{cor:thetatilde}, claim \eqref{q2} holds. 

For $n\le 3$, claims \eqref{q2} and \eqref{q3} follow from the $1$-formality 
of the groups $\vP^+_n$ and $\vP_n$ in that range, and Corollary \ref{cor:thetatilde}. 

Using now the algorithms described in Remarks \ref{remk:comp} 
and \ref{rem:comp-bis}, a Macaulay2 \cite{M2} computation reveals that 
\begin{align}
\label{eq:chencom}
\notag
\sum_{k\geq 2} \theta_k(\vP_4^+) t^{k-2} &=\sum_{k\geq 2} \theta_k(\fh(\vP_4^+)) t^{k-2}
= (8-3t+t^2)/(1-t)^4,\\
\sum_{k\geq 2} \theta_k(\vP_5^+) t^{k-2} &=\sum_{k\geq 2} \theta_k(\fh(\vP_5^+)) t^{k-2}
=(20+15t+5t^2)/(1-t)^4,\\ \notag
\sum_{k\geq 2} \theta_k(\vP_6^+) t^{k-2} &=\sum_{k\geq 2} \theta_k(\fh(\vP_6^+)) t^{k-2}
=(40+35t-40t^2-20t^3)/(1-t)^5.  
\end{align}
This establishes claim \eqref{q2}  for $4\le n\le 6$, thereby completing the proof. 
\end{proof}

It would be interesting to decide whether the equalities in parts \eqref{q2} 
and \eqref{q3} of the above proposition hold for all $n$ and all $k$.

Finally, let us address the validity of the Chen ranks formula 
\eqref{eq:Chenranksformula} for the pure virtual braid groups on $n\ge 4$ strings.  
We know from Lemma \ref{lem:resonancePV4+} that $\cR_1^1(vP_4^+)$ 
has a single irreducible component of dimension $4$.
A similar Macaulay2 \cite{M2} computation shows that 
$\cR_1^1(vP_5^+)$ has $15$ irreducible components, all of dimension $4$.
Using now \eqref{eq:chencom}, it is readily seen 
that the Chen ranks formula does not hold for either $\vP_4^+$ or $\vP_5^+$. 
Based on this evidence, and some further computations, we expect that 
the Chen ranks formula does not hold for any of the groups $vP_n$ and 
$vP_n^+$ with $n\geq 4$.

\newcommand{\arxiv}[1]
{\texttt{\href{http://arxiv.org/abs/#1}{arXiv:#1}}}
\newcommand{\doi}[1]
{\texttt{\href{http://dx.doi.org/#1}{doi:#1}}}
\renewcommand*\MR[1]{%
\StrBefore{#1 }{ }[\temp]%
\href{http://www.ams.org/mathscinet-getitem?mr=\temp}{MR#1}}

\end{document}